\documentclass[12pt]{amsart}
\usepackage{float} 
\usepackage{amsmath,amsthm,amssymb,enumitem,verbatim,stmaryrd,xcolor,microtype,graphicx,aliascnt,needspace,mathtools}
\usepackage[T1]{fontenc}
\usepackage[utf8]{inputenc}
\usepackage[english]{babel} 
\usepackage[top=3.5cm,bottom=3.55cm,left=3.5cm,right=3.5cm]{geometry}
\usepackage[bookmarksopen=true,bookmarksopenlevel=2,bookmarksdepth=2,linktoc=page,colorlinks,linkcolor={red!70!black},citecolor={red!80!black},urlcolor={blue!80!black},pdfpagemode=UseOutlines,pdfstartview={fitH}]{hyperref}
\usepackage{longtable}

\usepackage{tikz}
\usepackage{tikz-cd}
\usepackage[mathcal]{euscript}                               

\usepackage[colorinlistoftodos,bordercolor=orange,backgroundcolor=orange!20,linecolor=orange,textsize=footnotesize]{todonotes} \makeatletter \providecommand \@dotsep{5} \def\listtodoname{List of Todos} \def\listoftodos{\@starttoc{tdo}\listtodoname} \makeatother 

\usepackage{newtxtext,newtxmath}                              

\allowdisplaybreaks 

\usepackage{etoolbox}
\makeatletter
\patchcmd{\@startsection}{\@afterindenttrue}{\@afterindentfalse}{}{}             
\patchcmd{\part}{\bfseries}{\bfseries\LARGE}{}{}
\patchcmd{\section}{\scshape}{\bfseries}{}{}\renewcommand{\@secnumfont}{\bfseries} 
\patchcmd{\@settitle}{\uppercasenonmath\@title}{\large}{}{}
\patchcmd{\@setauthors}{\MakeUppercase}{}{}{}
\addto{\captionsenglish}{} 
\addto{\captionsenglish}{} 
\addto{\captionsenglish}{} 
\makeatother

\usepackage{fancyhdr}

\addto\extrasenglish{}  
\theoremstyle{plain}
\newtheorem{thm}{Theorem}[section] 
\newaliascnt{lemma}{thm}\newtheorem{lemma}[lemma]{Lemma}\aliascntresetthe{lemma}
\newaliascnt{cor}{thm}\newtheorem{cor}[cor]{Corollary}\aliascntresetthe{cor}
\newaliascnt{prop}{thm}\newtheorem{prop}[prop]{Proposition}\aliascntresetthe{prop}
\newaliascnt{claim}{thm}\aliascntresetthe{claim}
\newtheorem{thmA}{Theorem} 
\newaliascnt{propA}{thmA}\newtheorem{propA}[propA]{Proposition}\aliascntresetthe{propA}
\newaliascnt{lemmaA}{thmA}\aliascntresetthe{lemmaA}
\newtheorem*{claim*}{Claim}
\newtheorem*{thm*}{Theorem}
\newtheorem*{lem*}{Lemma}
\newtheorem*{cor*}{Corollary}

\theoremstyle{definition}
\newaliascnt{df}{thm}\newtheorem{df}[df]{Definition}\aliascntresetthe{df}
\newaliascnt{rem}{thm}\newtheorem{rem}[rem]{Remark}\aliascntresetthe{rem}
\newaliascnt{ex}{thm}\newtheorem{ex}[ex]{Example}\aliascntresetthe{ex}
\newaliascnt{conj}{thm}\aliascntresetthe{conj}
\newaliascnt{problem}{thm}\aliascntresetthe{problem}

\newtheorem*{df*}{Definition}
\newtheorem*{ex*}{Example}
\newtheorem*{rem*}{Remark}

\theoremstyle{remark}

\DeclareRobustCommand{\gobblefour}[5]{}    

\DeclareMathOperator{\Gr}{Gr}

\DeclareMathOperator{\rk}{rk}
\DeclareMathOperator{\cork}{cork}

\DeclareMathOperator{\codim}{codim}
\DeclareMathOperator{\sign}{sign}
\DeclareMathOperator{\upZ}{Z}

\newcommand{\K}{{\mathbb K}}

\newcommand{\N}{{\mathbb N}}

\renewcommand{\P}{{\mathbb P}}

\newcommand{\R}{{\mathbb R}}

\newcommand{\T}{{\mathbb T}}

\newcommand{\bV}{{\mathbf V}}
\newcommand{\cA}{{\mathcal A}}
\newcommand{\cB}{{\mathcal B}}
\newcommand{\cC}{{\mathcal C}}

\newcommand{\cH}{{\mathcal H}}

\newcommand{\cL}{{\mathcal L}}

\newcommand{\cP}{{\mathcal P}}

\newcommand{\cS}{{\mathcal S}}

\newcommand{\cV}{{\mathcal V}}

\newcommand{\br}{{\mathbf r}}

\newcommand{\id}{\textup{id}}

\renewcommand{\max}{\textup{max}}

\renewcommand{\geq}{\geqslant}
\renewcommand{\leq}{\leqslant}
\newcommand{\gen}[1]{\langle #1 \rangle}

\renewcommand{\setminus}{\backslash}

\renewcommand{\emptyset}\varnothing

\title{Flats and hyperplane arrangements for matroids with coefficients}

\author{Jannis Koulman}
\address{\rm Jannis Koulman, University of Groningen, the Netherlands}
\email{j.w.koulman@student.rug.nl}

\author{Oliver Lorscheid}
\address{\rm Oliver Lorscheid, University of Groningen, the Netherlands}
\email{o.lorscheid@rug.nl}

\hypersetup{pdftitle={Flats and hyperplane arrangements for matroids with coefficients},pdfauthor={Jannis Koulman and Oliver Lorscheid}}

\begin{document}

\begin{abstract}
 Based on the notion of vectors and linear subspaces for a matroid, we develop a theory of flats and hyperplane arrangements for $T$-matroids, where $T$ is a tract. This leads to several cryptomorphic descriptions of $T$-matroids: in terms of its lattice of $T$-flats, as a hyperplane arrangement over $T$, as point-line arrangements in projective space over $T$ and as a quiver representation over $T$. We examplify these notions in the case of tropical linear spaces, a.k.a.\ valuated matroids.
\end{abstract}

\maketitle

\begin{small} \tableofcontents \end{small}


\section*{Introduction}
\label{introduction}

\subsection*{Synopsis}
In the late 80ies, Dress developed in \cite{Dress86} a general framework of matroids with coefficients in a fuzzy ring. As a particular instance, this brought forth the notion of a valuated matroid (\cite{Dress-Wenzel92}), which was subsequently interpreted as a tropical linear space by Speyer (\cite{Speyer08}).

About a decade ago, Baker and Bowler (\cite{Baker-Bowler19}) reinvented and refurbished matroid theory with coefficients, by replacing the somewhat involved notion of a fuzzy ring by the simple and far reaching concept of a tract. 
This paper fell on fertile ground and led to a series of subsequent expansions, generalizations and applications (such as in \cite{%
Anderson19,
Anderson-Davis19,
BHKKL0,
BHKL1,
Baker-Jin23,
Baker-Lorscheid21,
Baker-Lorscheid25a,
Baker-Lorscheid25b,
Baker-Lorscheid-Zhang25,
Baker-Solomon-Zhang22,
Baker-Zhang23,
Bollen-Cartwright-Draisma22,
Bollen-Draisma-Pendavingh18,
Bowler-Pendavingh19,
Bowler-Su21,
Eppolito-Jun-Szczesny20,
Jarra23,
Jarra-Lorscheid24,
Jarra-Lorscheid-Vital24,
Jin-Kim25,
Kim24,
Pendavingh18,
Ting23,
Ting20}).

An important result for the present purposes is Anderson's axiomatic characterization of the vector sets of $T$-matroids in \cite{Anderson19}, where $T$ is a tract. This yields, in particular, the notion of a \emph{linear subspace}\footnote{Anderson uses the term \emph{$T$-vector set} for what we call a linear subspace of $T^n$.} of $T^n$, a striking insight which left, however, only few traces in the literature so far. If $T=K$ is a field, then this agrees with the usual notion of a linear subspace. In the case of the so-called \emph{tropical hyperfield $\T$}, a linear subspace is the same as a tropical linear space.

In the case of usual matroids $M$, which are canonically identified with $\K$-matroids (where $\K$ is the so-called \emph{Krasner hyperfield}), a linear subspace of $\K^n$ corresponds in a certain precise sense to the lattice of flats of the matroid $M$ (cf.\ \autoref{ex: linear subspaces}). Turning this relation around yields an identification of the lattice of flats with a lattice of linear subspaces of $\K^n$. The generalization of this correspondence to arbitrary tracts is the purpose of this text.

Before we explain our findings in detail, we give a brief overview of the contents of this paper:
\begin{enumerate}
 \item We introduce the notion of $T$-flats of a matroid (\autoref{subsection: lattices of T-flats}) and provide a cryptomorphic characterization of $T$-matroids in terms of their lattice of $T$-flats (\autoref{thm: correspondence between T-matroids lattices of T-flats}).
 \item We characterize the linear dependency of hyperplane functions (a.k.a.\ cocircuits) in terms of matroid quotients (\autoref{lemma: characterization of linearly dependent tuples of hyperplane functions}) and derive a cryptomorphic characterization of $T$-matroids in terms of matroid quotients in small ranks (\autoref{thm: cryptomorphism through F-quotients of small corank}), including an interpretation as a point-line arrangement in projective space (\autoref{thm: correspondence between T-matroids and point-line arrangements}).
 \item We identify a hyperplane arrangement over a field $T=K$ with the orthogonal complements of certain $K$-flats of a certain associated $K$-matroid (\autoref{thm: hyperplane arrangements for fields and the associated K-matroid}), which leads to a generalization of hyperplane arrangements to arbitrary tracts (\autoref{subsection: the hyperplane arrangement of a T-matroid}), including an application to tropical hyperplane arrangements (\autoref{section: Application to valuated matroids}).
\end{enumerate}

\subsection*{Tracts and \texorpdfstring{$T$}{T}-matroids}

The reader is expected to be versed in usual matroid theory (with Oxley's book \cite{Oxley92} as a general reference). We provide all necessary background from Baker-Bowler theory in \autoref{section: Background on Baker-Bowler theory}, including Anderson's axiomatization of linear subspaces (in \autoref{subsection: linear subspaces}). For the purpose of this introduction, we restrict ourselves to a brief informal outline of this theory.

A \emph{tract} $T$ is a generalization of a field where addition is replaced by a collection of zero sums ``$a_1+\dotsb+a_n=0$.'' This provides a notion of linear dependency: $X_1,\dotsc,X_r\in T^n$ are \emph{linearly dependent} if there is a nonzero tuple $(a_j)\in T^r$ such that $a_1X_{i,1}+\dotsb+a_rX_{i,r}=0$ for all $i=1,\dotsc,n$. Analogously to fields, we require that every nonzero element of $T$ is multiplicatively invertible, i.e., $T^\times=T-\{0\}$ is an abelian group.

In the following, we fix a ground set $E=\{1,\dotsc,n\}$ and consider elements of $T^E=T^n$ at times as functions $X:E\to T$. Let $M$ be a matroid on $E$, $\Lambda$ its lattice of flats and $\cH\subset\Lambda$ the set of its hyperplanes, or corank $1$ flats. A \emph{$T$-representation of $M$} is a map $\eta:\cH\to T^E$ such that $H=\{e\in E\mid \eta_H(e)=0\}$ for all $H\in\cH$ and such that $\eta_{H_1}$, $\eta_{H_2}$ and $\eta_{H_3}$ are linearly dependent whenever $F=H_1\cap H_2\cap H_3$ is a corank $2$ flat of $M$. 

Two $T$-representations $\eta$ and $\theta$ are equivalent if there are elements $a_H\in T^\times$ such that $\eta_H=a_H\cdot\theta_H$ for all $H\in\cH$. A \emph{(weak) $T$-matroid} is an equivalence class $[\eta]$ of $T$-representations of some matroid $M$, which we call the \emph{underlying matroid of $[\eta]$}.

In the main body of this paper, we also detail all notions and results from the perspetive of Grassmann-Pl\"ucker functions, which we omit from this introduction for simplicity.

\subsection*{Vectors and linear subspaces}

Two elements $X,Y\in T^n$ are \emph{orthogonal} if $X_1Y_1+\dotsb+X_nY_n=0$. The \emph{orthogonal complement $S^\perp$} of a subset $S\subset T^n$ is the collection of all $X\in T^n$ that are orthogonal to all elements of $S$. The \emph{vector set} of a $T$-matroid $[\eta:\cH\to T^n]$ is $\cV_\eta=\{\eta_H\mid H\in\cH\}^\perp$.

Let $V\subset T^n$ and $J\subset E$ be subsets. A \emph{normal form for $V$ (indexed by $J$)}\footnote{Also here we deviate from the terminology in \cite{Anderson19}, which uses the term \emph{reduced row-echolon form} for what we call normal form.} is a family $\{S_i\in V\mid i\in J\}$ of elements $S_i\in V$ such that $S_i(j)=\delta_{i,j}$ for all $i,j\in J$. A normal form $\{S_i\mid i\in J\}$ is \emph{maximal} if there exists no normal form $\{S'_i\mid i\in J'\}$ with $J\subsetneq J'$. Note that every normal form is linearly independent. We write $X\in\gen{S_i}$ if $\{X,S_i\}_{i\in J}$ is linearly dependent.

A \emph{linear subspace} of $T^n$ is a $T$-invariant subset $V\subset T^n$ that contains an element $X\in T^n$ if and only if $X\in\gen{S_i}$ for every maximal normal form $\{S_i\mid i\in B\}$ for $V$.


By \cite[Thm.\ 2.18]{Anderson19}, a subset $V\subset T^n$ is a linear subspace if and only if it equals the vector set of a $T$-matroid $[\eta]$. By \autoref{lemma: support bases of a vector set are the cobases of the underlying matroid}, $\{S_i\mid i\in B\}$ is a maximal normal form for $V$ if and only if $B$ is a cobasis of the underlying matroid of $\eta$. The cardinality of any such a $B$ yields thus a satisfactory notion of \emph{dimension} for $V$.

\begin{rem*}
 As mentioned above, linear subspaces agree with the usual notion from linear algebra in the case of a field $T=K$ and with tropical linear spaces in the case of the tropical hyperfield $T=\T$. As pointed out to the authors by Jeff Giansiracusa, the lattice of flats of a matroid $M$ corresponds bijectively to the covector set $\cV^\perp_\eta$ of the unique $\K$-matroid $[\eta]$ with underlying matroid $M$, where $\K=\{0,1\}$ is the Krasner hyperfield. This can be seen in terms of the bend relations for the Bergman fan of $M$ (as indicated to us by Giansiracusa, cf.\ \cite{Giansiracusa-Giansiracusa16}) or deduced from Anderson's result \cite[Prop.\ 5.2]{Anderson19}, cf.\ \autoref{ex: linear subspaces} for details.
\end{rem*}

Throughout the paper, we make the standing assumption that:
\begin{center}
 {\bf\textit{every tract is perfect,}}
\end{center}
which means that taking orthogonal complements defines an involution on the collection of all linear subspaces of $T^n$ (cf.\ \autoref{section: T-flats} for more details).

\subsection*{The lattice of \texorpdfstring{$T$}{T}-flats}
Let $[\eta:\cH\to T^n]$ be a $T$-matroid with underlying matroid $M$, $F\in\Lambda$ a flat of $M$ and $\cH_F=\{H\in\cH\mid F\subset H\}$. The \emph{$T$-flat of $[\eta]$ over $F$} is $\cV_F=\{\eta_H\mid H\in\cH_F\}^\perp$, which is a linear subspace of $T^n$ whose codimension equals the corank of $F$. More specifically, the restriction $\eta_F:\cH_F\to T^n$ of $\eta$ to $\cH_F$ defines a $T$-matroid $[\eta_F]$ whose underlying matroid is $M/F\oplus U_{0,F}$, where $U_{0,F}$ stands for the uniform rank $0$ matroid with ground set $F$, and $\cV_F=\cV_{\eta_F}$ the vector set of $[\eta_F]$. For example, if $F=H$ is a hyperplane, then the codimension of $\cV_H$ is $1$ accordingly we call linear subspaces of codimension $1$ \emph{$T$-hyperplanes}.

This yields a collection $\bV_\eta=\{\cV_F\mid F\in \Lambda\}$ of linear subspaces $\cV_F$ of $T^n$, which is partially ordered by inclusion and which we call the \emph{lattice of $T$-flats of $[\eta]$}. \autoref{thm: correspondence between T-matroids lattices of T-flats} characterizes lattices of flats axiomatically (also cf.\ \autoref{rem: alternative axioms for lattice of flats} for the axioms below):

\begin{thmA}\label{thmA}
 A collection $\bV$ of linear subspaces of $T^n$ is the lattice of $T$-flats of a $T$-matroid $[\eta]$ if and only if it satisfies that
 \begin{enumerate}
  \item $\bV$ contains $T^n$ and is closed under intersection
  \item every $\cV\in\bV$ is the intersection of $T$-hyperplanes in $\bV$
  \item every maximal linear subspace $\cV\in\bV-\{T^n,T\text{-hyperplanes}\}$ has codimension $2$.
 \end{enumerate}
\end{thmA}

Note that $T$-hyperplanes admit a simple characterization: they are the orthogonal complements of a single nonzero vector in $T^n$. A characterization of linear subspaces of codimension $2$ is more involved. In \autoref{thm: correspondence between T-matroids lattices of T-flats}, we provide, in fact, a more elementary (albeit more complicated looking) axiomatic characterization of lattices of $T$-flats.

The proof of \autoref{thmA} passes through the following characterization of a $T$-representation of $M$ (cf.\ \autoref{thm: cryptomorphism through F-quotients of small corank}).

\begin{thmA}\label{thmB}
 Let $M$ be a matroid and $\Lambda^{\leq2}$ be the collection of its flats of corank $1$ and $2$. Sending a $T$-representation $\eta$ to the collection $\Psi=\{\cV_F\mid F\in\Lambda^{\leq2}\}$ establishes a bijection between the $T$-matroids $[\eta]$ with underlying matroid $M$ and the collections $\{V_F\mid F\in\Lambda^{\leq2}\}$ of linear subspaces of $T^n$ with the properties:
 \begin{enumerate}
  \item the codimension of $V_F$ is equal to the corank of $F$
  \item $V_F\subset V_H$ whenever $F\subset H$.
 \end{enumerate}
\end{thmA}

We would like to draw the reader's attention to the fact that both \autoref{thmA} and \autoref{thmB} characterize $T$-matroids in terms of conditions for linear subspaces of codimension $1$ and $2$ only, which reflects the profound result \cite[Thm.\ 3.46]{Baker-Bowler19} that over a perfect tract, every weak $T$-matroid (which is characterized by conditions in coranks $1$ and $2$) is strong (properties in all coranks).

\subsection*{Point-line arrangements}
Taking orthogonal complements and projectivization of $T$-flats of small corank identifies $T$-matroids with point-line arrangements in \emph{projective space} $\P(T^n)=\big(T^n-\{0\}\big)T^\times$, whose elements $[a_e]_{e\in E}$ are $T^\times$-orbits of nonzero tuples $(a_e)\in T^n$. A \emph{line $\overline L$ in $\P(E)$} is the collection of the nonzero $T^\times$-orbits of a $2$-dimensional linear subspace $L$ of $T^n$.

A \emph{point-line arrangement in $\P(T^n)$} is a finite collection of \emph{points} (or elements) of $\P(T^n)$ together with a finite collection $\cL$ of lines in $\P(T^n)$ such that
\begin{enumerate}
 \item every line in $\cL$ contains at least two points from $\cP$
 \item $\cC^\ast=\{\underline X\mid X\in\cP\}$ is the cocircuit set of a matroid $M$
 \item every \emph{modular} pair of points $X,Y\in\cP$ (i.e., $E-(\underline X\cup\underline Y)$ is a corank $2$ flat of $M$) is contained in a uniquely determined line in $\cL$.
\end{enumerate}
Here $\underline X=\{e\in E\mid a_e\neq0\}$ is the support of $X=[a_e]_{e\in E}$.

\begin{thmA}\label{thmC}
 There is a canonical bijection
 \[
  \big\{\text{$T$-matroids on $E$}\big\} \quad \longrightarrow \quad \big\{\text{point-line arrangements in $\P(T^n)$}\big\}
 \]
 that sends a $T$-matroid $[\eta]$ with underlying matroid $M$ to the point-line arrangement $(\cP,\cL)$ where $\cP$ is the collection of all points $[a_e]\in\P(T^n)$ for which $(a_e)^\perp=\cH_H$ for some hyperplane $H$ of $M$ and where $\cL$ is the collection of lines $\overline L$ where $L=\cV_F^\perp$ for some corank $2$ flat $F$ of $M$.
\end{thmA}

\subsection*{Hyperplane arrangements}

Let $K$ be a field. A \emph{hyperplane arrangement} in a $K$-vector space $V$ is a finite collection $\cA=\{H_i\mid i\in E\}$ of pairwise distinct hyperplanes (or codimension $1$ subspaces) $H_i\subset V$. The hyperplanes arrangement determines a matroid $M_\cA$ whose flats are those subsets $F\subset E$ for which $j\in F$ whenever $\bigcap_{i\in F}H_i\subset H_j$. We assume that the intersection of all hyperplanes $H_i\in\cA$ is trivial, so the dimension $r$ of $V$ coincides with rank $M_\cA$.

We warn the reader of the following  (historically grown) incompatibility of terminology: a hyperplane $H_i$ of $V$ corresponds to the rank $1$ flat $\{i\}$ of the lattice of flats of $M_\cA$, which is thus not a hyperplane of $M_\cA$ (unless $r=2$).

Choosing linear functionals $\lambda_i:V\to K$ with $\ker\lambda_i=H_i$ for every $i\in E$ yields an embedding $\lambda:V\to K^n$ and a $K$-matroid $[\eta]$ with covector set $\cV_\eta^\perp=\lambda(V)$. The rescaling class of $[\eta]$ is the collection of all $K$-matroids $[\theta]$ that have the same underlying matroid $M$ as $[\eta]$ and for which there are $t_e\in T^\times$ with $\eta_H(e)=t_e\theta_H(e)$ for all $H\in\cH$ and $e\in E$, where $\cH$ is the hyperplane set of $M$.
The following is \autoref{thm: hyperplane arrangements for fields and the associated K-matroid}.

\begin{thmA}\label{thmD}
 Let $\cA$ be a hyperplane arrangement in $V$, $\lambda:V\to K^n$ as above and $[\eta]$ a $K$-matroid with $\cV_\eta^\perp=\lambda(V)$. Then the following hold:
 \begin{enumerate}
  \item The rescaling class of $[\eta]$ is independent of the choice of the $\lambda_i$.
  \item The underlying matroid $M$ of $[\eta]$ is equal to $M_\cA$.
  \item For every flat $F$ of $M$, the $K$-flat $\cV_F$ of $[\eta]$ over $F$ satisfies
  \[
   \cV^\perp_F \ = \ \bigcap_{i\in F} \ \lambda(H_i).
  \]
  In particular, $\lambda(H_i)=\cV^\perp_{\{i\}}$.
  \item Every $K$-matroid $[\eta]$ comes from a hyperplane arrangement in this way namely for $V=\cV^\perp_\eta\subset K^n$ and $\cA=\{H_i\mid i\in E\}$ with
  \[
   H_i \ = \ \{X\in V\mid X(i)=0\}.
  \]
 \end{enumerate}
\end{thmA}

This reinterpretation of hyperplane arrangements in terms of intersections of linear subspaces $V\subset K^n$ with corrdinate hyperplanes allows for a generalization to arbitrary tracts $T$, which do not possess a notion of \emph{abstract} linear spaces. This leads to the following notion.

Let $[\eta:\cH\to T^n]$ be a $T$-matroid and assume that its underlying matroid $M$ is simple. The \emph{canonical hyperplane arrangement of $[\eta]$} is the collection $\cA=\{H_i\mid i\in E\}$ of subsets $H_i=\{X\in\cV^\perp_\eta\mid X(i)=0\}$ of the covector set $\cV^\perp_\eta$ of $[\eta]$. The following is \autoref{prop: canonical hyperplane arrangement}.

\begin{propA}\label{propE}
  Let $\cA=\{H_i\mid i\in E\}$ be the canonical hyperplane arrangement of $[\eta]$. For $S\subset E$, let $H_S=\bigcap_{i\in S} H_i$. Then $F\subset E$ is a flat of $M$ if and only if $F=\{i\in E\mid H_F\subset H_i\}$. In this case, $H_F=\cV^\perp_F$. In particular, $H_F\subset T^n$ is a linear subspace of dimension $\cork F$.
\end{propA}





\section{Background on Baker-Bowler theory}
\label{section: Background on Baker-Bowler theory}

In this section, we recall the necessary  background from Baker and Bowler's paper \cite{Baker-Bowler19} on matroid theory with coefficients in a tract.

\subsection{Tracts}
\label{subsection: tracts}

The central algebraic structure in this theory is a tract, which can be thought of as a generalization of a field. 

A \emph{pointed group} is a commutative and multiplicative written monoid $T$ with an element $0\in T$ (called its \emph{zero}) such that $0\cdot a=0$ for all $a\in T$ and such that its \emph{unit group} is $T^\times=T-\{0\}$. 

The \emph{ambient semiring} of a pointed group $T$ is the group semiring $T^+=\N[T^\times]$. Sending $0$ to the additive neutral element of $T^+$ and sending nonzero elements $a\in T^\times$ to $1.a\in T^+$ embeds $T$ into $T^+$. This allows us to consider elements of $T$ as elements of $T^+$ and, conversely, to write every element of $T^+$ as a sum $\sum a_i$ of finitely many elements $a_i$ of $T$. An \emph{ideal} of $T^+$ is a subset $I$ that contains $0$ and that is closed under addition and multiplication by $T^+$.

An \emph{(ideal) tract}\footnote{In favour of a simplified account, we omit the attribute ``ideal'', which indicates that the null set is an ideal, an assumption that is not required in the original definition of a tract in \cite{Baker-Bowler19} (the nullset is not required to be closed under addition). In some papers, ideal tracts are called idylls. Note that all perfect tracts are ideal (\cite{Baker-Bowler-Lorscheid}), so in later parts of the papers, this discrepancy in terminology will vanish.} is a pointed group $T$ together with an ideal $N_T$ of $T^+$ (called its \emph{nullset}) such that for every $a\in T$, there is a unique $b\in T$ with $a+b\in N_T$. We call this element $b$ the \emph{additive inverse of $a$} and denote it by $-a$.

Immediate consequences of the axioms of a tract are that $N_T\cap T=\{0\}$, that $(-1)^2=1$, where $-1$ is the additive inverse of $1$, and that $-a=(-1)\cdot a$ for all $a\in T$.

A \emph{tract morphism} is a multiplicative map $f:T\to T'$ between tracts with $f(0)=0$, $f(1)=1$ and $\sum f(a_i)\in N_{T'}$ for all $\sum a_i\in N_T$.

\begin{ex}\label{ex: tracts}
 Every field $K$ is naturally a tract when we replace its addition by the null set $N_K=\{\sum a_i\in K^+\mid \sum a_i=0\text{ as sum in $K$}\}$. 
 
 The \emph{Krasner hyperfield} is the pointed group $\K=\{0,1\}$ (with the natural multiplication) together with nullset $N_\K=\K^+-\{1\}=\{n.1\mid n\neq1\}$. Note that $-1=1$.
 
 The \emph{tropical hyperfield} is the pointed group $\T=\R_{\geq0}$ (with the usual multiplication of real numbers) with nullset
 \[\textstyle
  N_\T \ = \ \big\{ \sum a_i\in \T^+ \, \big| \, \text{the maximum is assumed at least twice among the $a_i$'s}\big\}.
 \]
 
 The evident inclusion $\K\to\T$ is a tract morphism. Every tract $T$ admits a unique tract morphism $T\to\K$, which sends every nonzero element $a\in T$ to $1\in\K$. A tract morphism from $K$ to $\T$ is the same thing as a non-archimedean absolute value cf.\ \cite[Thm.\ 2.2]{Lorscheid22}. 
\end{ex}

\subsection{Grassmann-Plücker functions}
\label{subsection: Grassmann-Plücker functions}

Let $T$ be a tract, $E$ a finite set and $r\geq0$ an integer for the remainder of this section. 

A \emph{(weak) Grassmann-Plücker function of rank $r$ on $E$ with coefficients in $T$} is a function $\varphi:E^r \to T$ with:
\begin{enumerate}[label={(GP\arabic*)}]
 \item \label{GP1} 
 $\cB_\varphi=\big\{ \{x_1,\dotsc,x_r\} \, \mid \, \varphi(e_1,\dotsc,e_r)\neq0\big\}$ is nonzero and satisfies the basis exchange axiom
 \item \label{GP2} 
 $\varphi$ is \emph{alternating}, i.e., $\varphi(e_{\sigma(1)},\dotsc,e_{\sigma(r)})=\sign(\sigma)\cdot\varphi(e_1,\dotsc,e_r)$ for all $\sigma\in S_r$ and $\varphi(e_1,\dotsc,e_r) = 0$ if $\#\{e_1,\dotsc,e_r\}<r$ 
 \item \label{GP3}\label{GP-eq}  
 $\varphi$ satisfies the \emph{$3$-term Pl\"ucker relations}
 \begin{multline*}
   \varphi(e_1,\dotsc,e_{r-2},a,b) \ \cdot \ \varphi(e_1,\dotsc,e_{r-2},c,d) \\ - \ \varphi(e_1,\dotsc,e_{r-2},a,c) \ \cdot \ \varphi(e_1,\dotsc,e_{r-2},b,d) \\
   + \ \varphi(e_1,\dotsc,e_{r-2},a,d) \ \cdot \ \varphi(e_1,\dotsc,e_{r-2},b,c) \quad \in \quad N_T,
 \end{multline*}
 for all $e_1,\dotsc,e_{r-2},a,b,c,d\in E$.
\end{enumerate}

The unit group $T^\times$ acts on the set of all Grassmann-Pl\"ucker functions. A \emph{(weak) $T$-matroid of rank $r$ on $E$} is the $T^\times$-class $[\varphi]$ of a Grassmann-Pl\"ucker function $\varphi:E^r\to T$. 

The \emph{underlying matroid} of a $T$-matroid $[\varphi]$ is the matroid $M_\varphi$ with basis set $\cB_\varphi$, which is indeed a matroid by \ref{GP1} and independent of the choice of representative $\varphi$ for $[\varphi]$. In the case of the Krasner hyperfield $\K$, the map $[\varphi]\mapsto M_\varphi$ defines a bijection
\[
 \big\{ \text{$\K$-matroids} \big\} \ \longrightarrow \ \big\{ \text{matroids} \big\},
\]
since $\K^\times=\{1\}$ is a singelton (cf.\ \cite[Section 3]{Baker-Bowler19}), which can be considered as a primary example of a matroid with (trivial) coefficients. 

In the case of a field $K$, a $K$-matroid corresponds to a point of the Grassmannian $\Gr(r,E)(K)$ of $r$-dimensional linear subspaces of $K^E$ (cf.\ \cite[Ex.\ 3.30]{Baker-Bowler19}). A $\T$-matroid is the same as a valuated matroid or, equivalently, a tropical linear space (cf.\ \cite[Ex.\ 3.32]{Baker-Bowler19}).

\begin{rem}\label{rem: strong T-matroids}
 Since we restrict the presentation of the results of this paper to strong matroids, there is no difference between \emph{strong} and weak $T$-matroids, which allows us to omit the consideration of strong $T$-matroids altogether.
\end{rem}

\subsection{\texorpdfstring{$T$}{T}-representations of matroids}
\label{subsection: T-representations}

Let $M$ be a matroid of rank $r$ on $E$ for the rest of the paper. We denote by $\Lambda=\Lambda_M$ its lattice of flats and by $\cH=\cH_M$ its collection of hyperplanes, or flats of corank $1$. A \emph{modular triple of hyperplanes} is a triple $(H_1,H_2,H_3)$ of pairwise distinct hyperplanes that intersect in a corank $2$ flat $F=H_1\cap H_2\cap H_3$.

We consider elements $X\in T^E$ as functions $X:E\to T$ and denote their \emph{support} as $\underline X=\{e\in E\mid X(e)\neq0\}$. A \emph{hyperplane function} for $H\in\cH$ is a function $\eta_H:E\to T$ with support $E-H$. 

We say that $X_1,\dotsc,X_k\in T^E$ are \emph{linearly dependent} if there are $c_1,\dotsc,c_k\in T$ that are not all equal to $0$ such that
\[
 c_1X_1(e) \ + \ \cdots+ \ c_kX_k(e) \ \in \ N_T
\]
for all $e\in E$.

A \emph{$T$-hyperplane representation of $M$}, or for short a \emph{$T$-representation of $M$}, is a map $\eta:\cH\to T^E$ such that
\begin{enumerate}[label={(R\arabic*)}]
 \item \label{R1}
 $\eta_H$ is a hyperplane function for $H$, i.e., has support $E-H$, for every $H\in\cH$
 \item \label{R2}
 $\eta_{H_1}$, $\eta_{H_2}$ and $\eta_{H_2}$ are linearly dependent for every modular triple $(H_1,H_2,H_3)$ of hyperplanes.
\end{enumerate}

Since a hyperplane of $M$ is the same thing as the complement of a cocircuit, a comparison with the notion of a (weak) circuit signatures from \cite{Baker-Lorscheid25a} reveals that a map $\eta:\cH\to T^E$ is a $T$-representation of $M$ if and only if $\cC^\ast_\eta=\{a\eta_H\in T^E\mid a\in T^\times,H\in\cH\}$ is the $T$-circuit set of a $T$-matroid with underlying matroid $M$.

The result \cite[Thm.\ 2.16]{Baker-Lorscheid25a} provides a profound relation between Grassmann-Pl\"ucker functions and $T$-circuit sets or, in our case, $T$-representations. Let $\gen{-}$ be the closure operator of $M$, i.e., $\gen{S}$ is the smallest flat of $M$ containing $S\subset E$.

\begin{thm}\label{thm: correspondence between GP-functions and hyperplane functions}
 Let $\eta:\cH\to T^E$ be a $T$-representation of $M$. Then there is a unique $T$-matroid $[\varphi]$ with underlying matroid $M_\varphi=M$ that satisfies the \emph{fundamental relation} 
 \begin{equation*}\label{eq: fundamental relation}
  \frac{\eta_H(d)}{\eta_H(e)} \ = \ \frac{\varphi(i_1,\dots,i_{r-1},d)}{\varphi(i_1,\dots,i_{r-1},e)}
 \end{equation*}
 for every $H\in\cH$, for all $i_1,\dotsc,i_{r-1}\in H$ with $H=\gen{i_1,\dotsc,i_{r-1}}$ and $d,e\in E-H$. Conversely, every $T$-matroid comes from a $T$-representation in this way. Two $T$-representations $\eta$ and $\theta$ of $M$ correspond to the same $T$-matroid if and only if there are $a_H\in T^\times$ with $\eta_H=a_H\cdot\theta_H$ for all $H\in\cH$.
\end{thm}

We write $[\eta]$ for the class of a $T$-representation $\eta:\cH\to T^E$ modulo multiplying each hyperplane function $\eta_H$ with a nonzero scalar $a_H\in T^\times$. By \autoref{thm: correspondence between GP-functions and hyperplane functions}, the equivalence classes $[\eta]$ of $T$-representations $\eta$ of $M$ stay in bijection with the $T$-matroids $[\varphi]$ with underlying matroid $M_\varphi=M$. We write $[\eta]=[\varphi]$ when $\eta$ corresponds to $[\varphi]$.

The fundamental relation in \autoref{thm: correspondence between GP-functions and hyperplane functions} yields the notion of the \emph{fundamental hyperplane function} $\eta_H:E\to T$, with $\eta_H(e)=\varphi(i_1,\dotsc,i_{r-1},e)$ for every choice of a Grassmann-Pl\"ucker function $\varphi$ and of $i_1,\dotsc,i_{r-1}\in H$ that span $H$. A choice of a spanning set for each $H\in\cH$ yields thus a $T$-representation $\eta:\cH\to T^E$ that corresponds to $[\varphi]$. Conversely, every hyperplane function is the multiple of a fundamental hyperplane function:

\begin{cor}\label{cor: every hyperplane function is a multiple of a fundamental hyperplane function}
 Let $[\varphi:E^r\to T]$ be a $T$-matroid with underlying matroid $M$ and $\eta:\cH\to T^E$ a $T$-representation of $M$ with $[\eta]=[\varphi]$. Let $H\in\cH$ with $H=\gen{i_1,\dotsc,i_r}$. Then there is an $a\in T^\times$ such that $\eta_H(e)=a\varphi(i_1,\dotsc,i_r,e)$ for all $e\in E$.
\end{cor}

\begin{proof}
 This follows at once from the fundamental relation in \autoref{thm: correspondence between GP-functions and hyperplane functions}.
\end{proof}

\subsection{Duality}
\label{subsection: duality}

In order to define the dual of a $T$-matroid, we choose a total ordering of $E=\{e_1,\dotsc,e_n\}$ (where $n=\#E$) and define $\sign(e_1,\dotsc,e_n)$ as the sign of the permuation $\sigma\in S_n$ with the property that $\sigma(e_1)<\dotsc<\sigma(e_n)$.

Given a Grassmann-Pl\"ucker function $\varphi:E^r\to T$, we define $\varphi^\ast:E^{n-r}\to T$ by
\[
 \varphi^*(d_1,\dotsc,d_{n-r}) \ = \ \text{sign}(d_1,\dotsc,d_{n-r},e_1,\dotsc,e_{r})\varphi(e_1,\dotsc,e_r),
\]
whenever $E=\{d_1,\dotsc,d_{n-r},e_1,\dotsc,e_{r}\}$ (which does not depend on the order of $e_1,\dotsc,e_r$) and $\varphi^*(d_1,\dotsc,d_{n-r})=0$ if $\#\{d_1,\dotsc,d_{n-r}\}<n-r$. Note that $\varphi^\ast$ depends on the chosen ordering of $E$ only up to a global sign, which means that the \emph{dual $T$-matroid} $[\varphi]^\ast=[\varphi^\ast]$ is independent of the choice of this ordering. 

Duality for $T$-matroids comes with the expected properties (\cite[Theorem~3.24]{Baker-Bowler19}):
\begin{itemize}
 \item The dual of $[\varphi]$ is represented by $\varphi^\ast$: $[\varphi]^\ast=[\varphi^\ast]$.
 \item The underlying matroids of $[\varphi]^\ast$ and $[\varphi]$ are dual to each other: $M_{\varphi^\ast}=M_\varphi^\ast$.
\end{itemize}

\subsection{Minors}
\label{subsection: minors}

Another aspect of matroid theory that carries over to $T$-matroids are minors, which are defined as follows.

Let $\varphi:E^r\to T$ be a Grassmann-Pl\"ucker function. Let $S\subset E$ be a subset of rank $s$ and $\{i_1,\dotsc,i_s\}$ a spanning set (or maximal independent subset) of $S$. The \emph{contraction of $A$ in $[\varphi]$} is the $T$-matroid $[\varphi]/S=[\varphi/S]$ where $\varphi/S:(E-S)^{r-s}\to T$ is the Grassmann-Pl\"ucker function defined by 
\[
 \varphi/S(e_1,\dotsc,e_{r-s}):=\varphi(i_1.\dotsc,i_s,e_1,\dotsc,e_{r-s}).
\]
By \cite[Lemma 4.4]{Baker-Bowler19}, the $[\varphi]/S$ does not depend on the choice of $\{i_1,\dotsc,i_s\}$. In terms of $T$-representations, contractions assume the following shape (cf.\ \cite[Lemma 4.7]{Baker-Bowler19}).

\begin{prop}\label{prop: hyperplane functions of minors}
 Let $\eta:\cH\to T^E$ be a $T$-representation of $M$ with corresponding $T$-matroid $[\varphi]$. Let $S\subset E$ and $\cH_S=\{H\in\cH\mid S\subset H\}$. Then $\tilde\eta:\cH_S\to T^{E-S}$, given by the restrictions $\tilde\eta_H=\eta_H\vert_{E-S}$, is a $T$-representation whose corresponding $T$-matroid is $[\varphi]/S$.
\end{prop}

\begin{rem}\label{rem: deletion}
 The \emph{deletion of $S$ in $[\varphi]$} is defined as $[\varphi]\setminus S=([\varphi]^\ast/S)^\ast$, which can also be expressed in terms of an explicit formula, analogously to contractions. Since we do not use deletions in this paper, we refrain from a further discussion and refer the reader to \cite[Def.~4.3]{Baker-Bowler19} for further details.
\end{rem}

\subsection{Direct sums}
\label{subsection: direct sums}

$T$-representations of the direct sum $M_1\oplus M_2$ of two matroids $M_1$ and $M_2$ have been studied in \cite[Section 4]{Baker-Lorscheid25a}: they correspond to a pair of $T$-representations $[\eta_1]$ of $M_1$ and $[\eta_2]$ of $M_2$. However obvious, direct sums of $T$-matroids have not yet been introduced in the literature yet, as far as we can tell. We introduce this concept in this section. For the relation to \cite{Baker-Lorscheid25a}, see \autoref{rem: direct sums of T-representations}.

Given two Grassmann-Pl\"ucker function $\varphi:E^r\to T$ and $\psi:D^s\to T$ with disjoint ground sets, we define the \emph{direct sum of $[\varphi]$ and $[\psi]$} as the $T$-matroid with Grassmann-Pl\"ucker function $\varphi\oplus\psi:(E\sqcup D)^{r+s}\to T$ given by
\[
 \varphi\oplus\psi(e_1,\dotsc,e_r,d_1,\dotsc,d_s) \ = \ \varphi(e_1,\dotsc,e_r) \ \cdot \ \psi(d_1,\dotsc,d_s)
\]
for $e_1,\dotsc,e_r\in E$ and $d_1,\dotsc,d_s\in D$, which extends by \ref{GP2} to all permutations of the arguments, and $\varphi\oplus\psi(e_1,\dotsc,e_{r+s})=0$ if $\#\{i\mid x_i\in E\}\neq r$. The underlying matroid of $[\varphi\oplus\psi]$ is $M_{\varphi\oplus\psi}=M_\varphi\oplus M_\psi$. 

It is readily verified that $\varphi\oplus\psi$ is indeed a Grassmann-Pl\"ucker function and that its underlying matroid is $M_{\varphi\oplus\psi}=M_\varphi\oplus M_\psi$. Recall that the hyperplanes of $M_\varphi\oplus M_\psi$ are of the form $\hat H=H\cup D$ for $H\in\cH_{M_\varphi}$ and $\hat H=E\cup H$ for $H\in\cH_{M_\psi}$.

We apply this construction only in the case that $\psi=\varphi_{0,D}:D^0\to T$ is the Grassmann-Pl\"ucker function with $\varphi_{0,D}(0)=1$. In this case, the underlying matroid of $[\varphi_{0,D}]$ is $U_{0,D}$ and the hyperplanes of $M_\varphi\oplus M_{\varphi_{0,D}}=M\oplus U_{0,D}$ are of the form $H\sqcup D$ where $M=M_\varphi$ and $H$ is a hyperplane of $M$. 

This counterpart of this construction for $T$-representations is as follows: let $\eta:\cH_M\to T^E$ be a $T$-representation that corresponds to $[\varphi]$. Then we can derive a $T$-representation that corresponds to $[\varphi\oplus\varphi_{0,n}]$ by extending the hyperplane functions $\eta_H:E\to T$ by $0$ (according to axiom \ref{R1}) to functions $\hat\eta_H:E\sqcup D\to T$ for every hyperplane $H$ of $M$.

\begin{rem}\label{rem: direct sums of T-representations}
 Let $\eta:\cH_{M_\varphi}\to T^E$ and $\theta:\cH_{M_\psi}\to T^D$ be $T$-representations that correspond $[\varphi]$ and $[\psi]$, repectively. Let $\hat\eta_H:E\sqcup D\to T$ (resp.\ $\hat\theta_H:E\sqcup D\to T$) be the extension by $0$ of the hyperplane function $\eta_H:E\to T$ (resp.\ $\theta_H:D\to T$). Then the resulting map $\eta\oplus\theta:\cH_{M_{\varphi\oplus\psi}}\to T^{E\sqcup D}$ is a $T$-representation of $M_{\varphi\oplus\psi}$ that corresponds to $[\varphi\oplus \psi]$.
 
 This demonstrates that the direct sum of $T$-matroids is compatible with \cite[Thm.\ 4.1 Footnote 2]{Baker-Lorscheid25a}.
\end{rem}

\subsection{Vectors}
\label{subsection: vectors}

The vectors of a $T$-matroid are of central importance for our notion of a $T$-flat in \autoref{section: T-flats}. Vectors are introduced in by Baker and Bowler in \cite{Baker-Bowler19}, based on the ideas of Dress and Wenzel (\cite{Dress-Wenzel92b}), and they are thoroughly studied in Anderson's paper \cite{Anderson19}, including a cryptomorphic characterization of (strong) $T$-matroids in terms of vector axioms. 

Let $\varphi:E^r\to T$ be a Grassmann-Pl\"ucker function. The set of \emph{vectors of $[\varphi]$} is the subset 
\[\textstyle 
 \cV_\varphi \ = \ \big\{ X\in T^E \, \big| \, \sum_{e\in E} X(e)\cdot\varphi(i_1,\dotsc,i_{r-1},e)\in N_T\text{ for all $i_1,\dotsc,i_{r-1}\in E$}\big\}.
\]
of $T^E$, which does not depend on the choice of representative $\varphi$ for $[\varphi]$. The set of \emph{covectors of $[\varphi]$} is the set $\cV^\ast_\varphi=\cV_{\varphi^\ast}$.

We denote the \emph{orthogonal complement} of a subset $\cS$ of $T^E$ by $\cS^\perp=\{X\in T^E\mid X\perp Y\text{ for all }Y\in\cS\}$. The following result shows that our definition agrees with the more common definition of vectors as the orthogonal complement of the hyperplane functions (a.k.a.\ cocircuits).

\begin{prop}\label{prop: hyperplane description of vectors}
 Let $\eta:\cH\to T^E$ be a $T$-representation for $M$ corresponding to the $T$-matroid $[\varphi]$. Then $\cV_\varphi=\{\eta_H\mid H\in\cH\}^\perp$.
\end{prop}

\begin{proof}
 This follows at once from the fact that every hyperplane function $\eta_H$ is a multiple of a fundamental hyperplane function $\varphi(i_1,\dotsc,i_{r-1},-)$ given by \autoref{thm: correspondence between GP-functions and hyperplane functions}, and that $\varphi(i_1,\dots,i_{r-1},-)$
\end{proof}

In general, vectors are not orthogonal to covectors. A tract $T$ is called \emph{perfect} if $\cV_\varphi\perp\cV^\ast_\varphi$ for all $T$-matroids $[\varphi]$. Examples of perfect tracts are fields, $\K$ and $\T$.

An important property of perfect tracts is that every (weak) $T$-matroid $[\varphi]$ is strong (cf.\ \cite[Thm.\ 3.46]{Baker-Bowler19}), which means that $\varphi$ satisfies a larger collection of Pl\"ucker relations (analogous to the definining relations of a Grassmannian). We use this fact in the background, without the need to actually define strong $T$-matroids.

A consequence of the latter fact is that duality of (weak) $T$-matroids manifests itself as (strong) orthogonality (for perfect tracts). We say that two functions $X,Y\in T^E$ are \emph{orthogonal}, and write $X\perp Y$, if $\sum_{e\in E} X(e)\cdot Y(e)\in N_T$. The following is \cite[Thm.\ 3.26]{Baker-Bowler19}:

\begin{prop}\label{prop: orthogonality of hyperplanes and cohyperplanes}
 Let $T$ be perfect. Let $\eta:\cH\to T^E$ be a $T$-representation of $M$ with corresponding $T$-matroid $[\varphi]$ and $\eta^\ast:\cH\to T^E$ a $T$-representation of $M^\ast$ whose corresponding matroid is $[\varphi]^\ast$. Then $\eta_H\perp \eta^\ast_{C}$ for every hyperplane $H\in\cH_M$ and every cohyperplane $C\in\cH_{M^\ast}$.
\end{prop}

\subsection{Linear subspaces}
\label{subsection: linear subspaces}

Since we recover the circuit set $\cC^\ast_\varphi$ of a $T$-matroid $[\varphi]$ as the support minimal elements of $\cV_\varphi-\{0\}$, it is clear that the vector set $\cV_\varphi$ uniquely characterizes the $T$-matroid $[\varphi]$. Anderson provides in \cite{Anderson19} a characterization of those subsets $\cV\subset T^E$ that are equal to the vector set $\cV_\varphi$ of a $T$-matroid $[\varphi]$. We review this characterization in the following, but deviate slightly in presentation and terminology.

Let $\cV$ be a subset of $T^E$. A \emph{support basis for $\cV$} is an inclusion minimal subset $B\subset E$ with the property that for every $X\in\cV-\{0\}$ the intersection $\underline{X}\cap B$ is nonempty. We denote the collection of all support bases as $\cB^\ast_\cV$. (This notation is motivated by \autoref{lemma: support bases of a vector set are the cobases of the underlying matroid}, which shows that $\cB^\ast_\cV$ is the family of cobases of $M_\varphi$ if $\cV=\cV_\varphi$ for a $T$-matroid $[\varphi]$.)

Let $J\subset E$ be a subset. A \emph{normal form for $\cV$ (with respect to $J$)} is a subset $\{S_i\in\cV\mid i\in J\}$ of $\cV$ such that $S_i(j)=\delta_{i,j}$ for $i,j\in J$. If $B\in\cB^\ast_\cV$ is a support basis for $\cV$, then we call $\{S_i\in\cV\mid i\in B\}$ also a \emph{normal basis for $\cV$} for short. If $\cV$ is closed under scalar multiplication by $T$, which is the case if $\cV$ is a vector set, then $\cV$ has a $B$-basis for every $B\in\cB^\ast_\cV$ (cf.\ \cite[Lemma 2.4]{Anderson19}). Normal bases $\{S_i\mid i\in B\}$ are \emph{maximal} in the sense that if $\{S_i'\mid i\in J'\}$ is a normal form for $\cV$ with $B\subset J'$, then $B=J'$ (if there was a $j\in J'-B$, then $S_j'\cap B=\emptyset$).

The \emph{span of $\{S_i\mid i\in B\}$} is
\[\textstyle
 \gen{S_i\mid i\in B} \ = \ \big\{ X\in T^E \, \big| \, X(e)-\sum_{a\in B} X(a)\cdot S_a(e) \in N_T\text{ for all }a\in E-B\big\}.
\]
Note that $X\in\gen{S_i}$ if and only if $\{X,S_i\}_{i\in B}$ is linearly dependent.

\begin{df}
 A \emph{linear subspace of $T^E$} (or \emph{$T$-vector set} in the terminology of \cite{Anderson19}) is a subset $\cV\subset T^E$ that is $T$-invariant and that satisfies 
 \[
  \cV \ = \ \bigcap_{\substack{B\in\cB^\ast_\cV,\\ \{S_i\}\text{ $B$-basis for $\cV$}}} \gen{S_i\mid i\in B}.
 \]
\end{df}

The following is \cite[Thm.\ 2.18]{Anderson19}.

\begin{thm}\label{thm: Anderson's cryptomorphism for vector sets of T-matroids}
 A subset $\cV\subset T^E$ is a linear subspace if and only if it is the vector set of a matroid.
\end{thm}

\begin{ex}\label{ex: vector sets for T-matroids of rank 1}
 Any nonzero function $\varphi:E\to T$ is a Grassmann-Pl\"ucker function and define a $T$-matroid of rank $1$. The unique hyperplane of $M_\varphi$ is $H=\{e\in E\mid \varphi(e)=0\}$ and $\eta_H=\varphi$ is a hyperplane function that defines a $T$-representation $\eta:\{H\}\to T^E$ that corresponds to $[\varphi]$. The vector set of $\varphi$ is thus $\cV_\varphi=\eta_H^\perp$, defined by the unique hyperplane function $\eta_H$. 
 
 This shows that a subset $\cV\subset T^E$ is the vector set of a $T$-matroid of rank $1$ if and only if $\cV=\varphi^\perp$ for a nonzero function $\varphi:E\to T$ also cf.\ \cite[Prop.\ 5.1]{Anderson19}.
 
 It is not hard to verify that in the rank $1$ case that the support bases for $\cV_\varphi$ are of the form $B=E-e$ with $e\in E-H$ and that the condition $X(e)-\sum_{a\in E} X(a)\cdot S_a(e) \in N_T$ for a support $B$-basis $\{S_i\mid i\in B\}$ corresponds to $X\perp\eta_H$ since the orthogonality $S_i\perp\eta_H$ implies that $S_i(e)=-\eta_H(i)\eta_H(e)$. This recovers \autoref{thm: Anderson's cryptomorphism for vector sets of T-matroids} in the rank $1$ case.
\end{ex}

We complement the results of \cite{Anderson19} by the following observations.

\begin{lemma}\label{lemma: support bases of a vector set are the cobases of the underlying matroid}
 Let $\cV=\cV_\varphi$ be the vector set of a $T$-matroid $\varphi$ with underlying matroid $M=M_\varphi$. Then the set $\cB^\ast_\cV$ of support bases for $\cV$ is equal to the set $\cB^\ast_{M}$ of cobases of $M$.
\end{lemma}

\begin{proof}
 Let $\cC_\varphi\subset\cV$ be the set of $T$-circuits of $\varphi$, which are the support minimal $X\in\cV-\{0\}$, and let $\cC_M$ be the set of circuits of $M$. By \cite[Lemma 2.2]{Anderson19}, $\cB^\ast_\cV$ is equal to the set of support bases $\cB^\ast_{\cC_\varphi}$ for $\cC_\varphi$. By \cite[Lemma 3.6]{Baker-Bowler19}, $\cC_M=\{\underline{X}\mid X\in\cC_\varphi\}$.
 
 These identifications turn the definining property of a support basis $B\in\cB^\ast_\varphi$ into the property that $C\cap B\neq\emptyset$ for all $C\in\cC_M$. The minimality of $B$ means that for every $a\in B$, there is a $C\in\cC_M$ with $C\cap(B-a)=\emptyset$. Replacing $B$ by its complement $B^c=E-B$ turns these properties into $C\not\subset B$ for all $C\in\cC_M$, but for every $a\in E-B^c$, there is a $C\in\cC_M$ with $C\subset B^ca$. This characterizes $B^c$ as a maximal independent set of $M$, which means that $B$ is a cobasis of $M$, as claimed.
\end{proof}

The next result appears implicitly already in \cite{Anderson19}. Note that a circuit for $[\varphi]$ is a nonzero element $X\in\cV_\varphi$ with minimal support $\underline X$, which is the same thing as a hyperplane function $\eta^\ast_H$ of a $T$-representation $\eta^\ast:\cH_{M_\varphi^\ast}\to T^E$ for $[\varphi]^\ast$.

\begin{prop}\label{prop: rows of normal bases are circuits}
 Let $[\varphi]$ be a $T$-matroid and $\{S_i\mid i\in B\}$ be a normal basis for $\cV_\varphi$. Then $S_i:E\to T$ is a circuit of $[\varphi]$.
\end{prop}

\begin{proof}
 By assumption, $S_i$ is a vector of $[\varphi]$, so the composition with the unique tract morphism $f:T\to\K$ yields a vector $f\circ S_i$ of the $\K$-matroid $[f\circ\varphi]$. By \cite[Prop.\ 5.2]{Anderson19}, the support of $f\circ S_i$, which equals $\underline S_i$, is the union of circuits of $M_\varphi=M_{f\circ\varphi}$.
 
 By \autoref{lemma: support bases of a vector set are the cobases of the underlying matroid}, $B$ is a cobasis of $M_\varphi$, which means that its complement $B^c=E-B$ is a basis of $M_\varphi$. Thus $B^ci$ contains a unique circuit, which must consequently be equal to $\underline{S}_i$. This means that $S_i$ is a nonzero vector of $[\varphi]$ of minimal support and therefore a circuit of $[\varphi]$.
\end{proof}

\begin{ex}\label{ex: linear subspaces}
 If $T=K$ is a field, then a linear subspace of $K^E$ in the sense of this paper is nothing else than a linear subspace in the usual sense.
 
 If $T=\T$ is the tropical hyperfield, then a linear subspace $\cV\subset \T^E$ is the vector set of a $\T$-matroid, a.k.a.\ valuated matroid. A \emph{tropical linear space}\footnote{Strictly speaking, tropical linear spaces were introduced in \cite{Speyer08} as subspaces of the \emph{tropical torus} $(\T^\times)^E$, and the correspondence between linear subspaces of $\T^E$ and tropical linear spaces in $(\T^\times)^E$ requires that the valuated matroid is loopless. This technical restriction becomes obsolete if extend the notion of tropical linear spaces to covector sets in $\T^E$.} is, by definition, the covector set of a valuated matroid, which means that a linear subspace of $\T^E$ is the same as a tropical linear space. By \cite{Hampe15}, a tropical linear space in $\T^E$ is the same thing as a $\T$-submodule that is a tropical variety.
 
 By \cite[Prop.\ 5.2]{Anderson19}, the covector set of a matroid $\K$-matroid $[\varphi:E^r\to\K]$ with underlying matroid $M$ is characterized as 
 \[
  \cV^\ast_\varphi \ = \ \big\{X\in\K^E \, \big| \, \underline X\text{ is a union of cocircuits of $M$} \big\}.
 \]
 Note that since $\K^\times=\K-\{0\}=\{1\}$ is a singleton, $X\in\K^E$ is uniquely characterized by its support $\underline X$ and that $\varphi$ is uniquely determined by its underlying matroid $M$. Since the complement of a cocircuit is a hyperplane of $M$, the complement of the union of cocircuits is a flat of $M$. This shows that the bijection $\upZ:\K^E\to\cP(E)$ of $\K^E$ with the power set $\cP(E)$ of $E$, given by $\upZ(X)=\{e\in E\mid X(e)=0\}$, restricts to a bijection $\upZ:\cV^\ast_\varphi\to\Lambda_M$, where $\Lambda_M$ is the lattice of flats of $M$. 
 
 As a conclusion, we see that $\cV\subset\K^E$ is a linear subspace if and only if $\Lambda=\upZ(\cV)$ is the collection of flats of a matroid, which translates into the following three properties that characterize a linear subspace $\cV$ of $\K^E$:
 \begin{enumerate}
  \item $\cV$ contains the zero vector $0$
  \item given $X,Y\in\cV$, also $Z:e\mapsto\max\{X(e),Y(e)\}$ is in $\cV$
  \item for every nonzero $X\in\cV$, 
  \[
   \underline X \ = \ \coprod_{\substack{Y\in\cV\text{ with}\\ \underline{Y}\subsetneqq\underline{X}\text{ maximal}}} \underline{X}-\underline{Y}.
  \]
 \end{enumerate}
\end{ex}

\subsection{Matroid quotients}

The notion of matroid quotients was generalized in \cite{Jarra-Lorscheid24} from usual matroid theory to (strong) $T$-matroids. In order to avoid strong $T$-matroids, we assume that $T$ is perfect in this section, which implies that every weak $T$-matroid is strong.

Given two Grassmann-Pl\"ucker function $\varphi:E^r\to T$ and $\psi:E^s\to T$ on the same ground set $E$, with underlying matroid $M=M_\varphi$, we say that the $T$-matroid $[\varphi]$ is a \emph{quotient} of $[\psi]$ if $\{\eta_H\mid H\in\cH_M\}$ is a subset of $\cV^\ast_\psi$ where $\eta:\cH_M\to T^E$ is a $T$-representation for $[\varphi]$. 

\begin{thm}\label{thm: matroid quotients through vector inclusion}
    Let $T$ be a perfect tract and $E$ a finite set. Let $M$ and $N$ be two $T$-matroids over $E$ with covectors $\cV^*_M$ and $\cV^*_N$ respectively. Then $M$ is a quotient of $N$ if and only if $\cV^*_M\subset \cV^*_N$.
\end{thm}

\begin{proof}
    See \cite[Theorem~2.17]{Jarra-Lorscheid24}.
\end{proof}


\section{\texorpdfstring{$T$}{T}-flats}
\label{section: T-flats}

In this section, we introduce the notion of $T$-flats and study their properties. The following is the standing notation for the whole section: $T$ is a tract and $\varphi:E^r\to T$ is a Grassmann-Pl\"ucker function with underlying matroid $M=M_\varphi$. We denote by $\Lambda$ the lattice of flats of $M$ and by $\cH$ the collection of hyperplanes of $M$. We let $\eta:\cH\to T^E$ be a $T$-representation that corresponds to $[\varphi]$. 


From this point on, we assume that
\begin{center}
 {\it\textbf{$T$ is a perfect tract.}}
\end{center}
This means that vectors are orthogonal to covectors for all $T$-matroids. From a systematic perspective, this is a small restriction since every tract can be turned into a perfect tract by enlarging its null set (as explained in the forth-coming paper \cite{Baker-Bowler-Lorscheid}). This process of ``perfection'' preserves the class of (weak) $T$-matroids, including their Grassmann-Pl\"ucker functions and hyperplane functions, but it changes orthogonality and enlarges the set of vectors. Note that all examples of tracts that we discuss in this paper (fields, $\K$ and $\T$) are perfect. 


\subsection{\texorpdfstring{$F$}{F}-quotients}
\label{subsection: F-quotients}

Let $F\in\Lambda$ be a flat of corank $s$ and $\{i_1,\dotsc,i_{r-s}\}$ a spanning set for $F$. Let $\varphi/F:(E-F)^s\to T$ be the contraction of $F$ in $\varphi$ (cf.\ \autoref{subsection: minors}) and $\varphi_{0,F}:F^0\to T$ the Grassmann-Pl\"ucker function with $\varphi_{0,F}(0)=1$, whose underlying matroid is the uniform matroid $U_{0,F}$.

\begin{df}
 The \emph{$F$-quotient of $[\varphi]$} is the $T$-matroid $[\varphi_F]=[\varphi/F]\oplus[\varphi_{0,F}]$, which is represented by the Grassmann-Pl\"ucker function $\varphi_F:E^{s}\to T$ given by 
 \[
  \varphi_F(e_1,\dotsc,e_s) \ = \ \varphi(i_1.\dotsc,i_{r-s},e_1,\dotsc,e_s).
 \]
\end{df}

Note that $\varphi_F$ depends on the choice of $i_1,\dotsc,i_{r-s}\in F$ only up to a constant, and $[\varphi_F]$ is thus well-defined for $F$. The underlying matroid of $[\varphi_F]$ is $M_{\varphi_F}=M_{\varphi}/F\oplus U_{0,F}$ whose lattice flats is $\Lambda_F=\{F'\in\Lambda\mid F\subset F'\}$ and whose collection of hyperplanes is $\cH_F=\{H\in\cH\mid F\subset H\}$.

As trivial examples, we recover $[\varphi]=[\varphi_{\gen\emptyset}]$ as the \emph{trivial $F$-quotient}, where $F=\gen{\emptyset}$ is the bottom of $\Lambda$. If $F=E$, then $[\varphi_E]=[\varphi_{0,E}]$.

\begin{prop}\label{prop: T-representations for F-quotients}
 Let $F\in\Lambda$ be a flat of $M$ and $\eta:\cH\to T^E$ a $T$-reprentation for $[\varphi]$. Then its restriction $\eta_F:\cH_F\to T^E$ to $\cH_F\subset\cH$ is a $T$-representation for $[\varphi_F]$.
\end{prop}

\begin{proof}
 The hyperplanes of $MF$ are of the form $H-F$ where $H$ is a hyperplane of $M$ that contains $F$. The hyperplanes of $MF\oplus U_{0,F}$ add back the elements of $F$ as loops, i.e.\ the set of hyperplanes of $MF\oplus U_{0,F}$ is $\{H\in\cH\mid F\subset H\}=\cH_F$.
 
 Thus the restriction $\eta_F:\cH_F\to T^E$ of $\eta$ to $\cH_F$ is a collection of functions indexed by the hyperplanes of $MF\oplus U_{0,F}$, which inherits the properties of a $T$-representation from $\eta$.
\end{proof}

The term ``$F$-quotient'' for $[\varphi_F]$ finds its justification in the following fact.

\begin{cor}\label{cor: F-quotients are matroid quotients}
 Let $F\in\Lambda$ be a flat of $M$. Then $[\varphi_F]$ is a matroid quotient of $[\varphi]$. More generally, if $F\subset F'$ are two flats of $M$, then $[\varphi_{F'}]$ is a matroid quotient of $[\varphi_F]$.
\end{cor}

\begin{proof}
 Let $\eta:\cH\to T^E$ be a $T$-reprentation for $[\varphi]$. By \autoref{prop: T-representations for F-quotients}, its restriction $\eta_F:\cH_F\to T^E$ to $\cH_F$ is a $T$-representation for $[\varphi_F]$. By \autoref{prop: orthogonality of hyperplanes and cohyperplanes}, hyperplane functions $\eta_H$ are orthogonal to hyperplane functions $\eta^\ast_{C}$ (where $\eta^\ast:\cH_{M^\ast}\to T^E$ is a $T$-representation of $M^\ast$). By the definition of the covectors $\cV^\ast_\varphi$ as the elements $X\in T^E$ that are orthogonal to cohyperplane functions, we have thus
 \[
  \{\eta_H\in T^E \mid H\in\cH_F\} \quad \subset \quad \{\eta_H\in T^E\mid H\in\cH\} \quad \subset \quad \cV^\ast_\varphi,
 \]
 which shows that $[\varphi_F]$ is a matroid quotient of $[\varphi]$.
 
 It is evident from the definition of $[\varphi_{F}]$ is transitive in $F$, i.e., we have $[\varphi_{F'}]=[(\varphi_F)_{F'}]$. Thus applying the previously established claim to $\varphi$ replaced by $\varphi_F$ reveals $[\varphi_{F'}]$ as a matroid quotient of $[\varphi_F]$. 
\end{proof}

\begin{lemma}\label{lemma: hyperplane functions as F-quotients}
 Let $\eta:\cH\to T^E$ be a $T$-representation of $M$ and $[\varphi]$ the corresponding $T$-matroid. Then $[\eta_H]=[\varphi_H]$ for every $H\in\cH$.
\end{lemma}

\begin{proof}
 By definition, $\varphi_H(-)=\varphi(i_1,\dotsc,i_{r-1},-)$ is a fundamental hyperplane function, where $i_1,\dotsc,i_{r-1}$ is a spanning set for $H$. Thus the result follows from \autoref{cor: every hyperplane function is a multiple of a fundamental hyperplane function}.
\end{proof}

\subsection{Cryptomorphism through \texorpdfstring{$F$}{F}-quotients of ranks 1 and 2}
\label{subsection: On the linear dependency of hyperplane function}

In this section, we introduce a new characteriztion for $T$-matroids in terms of $T$-representations of hyperplanes and corank $2$ flats. The key result for this cryptomorphism is the following characterization of linear dependence for triples of hyperplane functions. Recall that we write $\cH_S=\{H\in\cH\mid S\subset H\}$ for $S\subset E$.


\begin{lemma}\label{lemma: characterization of linearly dependent tuples of hyperplane functions}
 Let $F$ be a corank $2$ flat of $M$ and $\eta:\cH_F\to T^E$ be a map that sends every $H\in\cH_F$ to a hyperplane function $\eta_H:E\to T$ for $H$. Then $\eta_{H_1}$, $\eta_{H_2}$ and $\eta_{H_3}$ are linearly dependent for every modular triple $(H_1,H_2,H_3)$ of hyperplanes of $M$ with $H_1\cap H_2\cap H_3=F$ if and only if there exists a $T$-matroid $[\varphi:E^2\to T]$ such that $[\eta_{H}]$ is a matroid quotient of $[\varphi]$ for every $H\in\cH_F$. 
 
 Such a $T$-matroid $[\varphi]$ is uniquely determined. If it exists, then its underlying matroid is $M_\varphi=(MF)\oplus U_{0,F}$.
\end{lemma}

\begin{proof}
 Assume that $\eta_{H_1},\eta_{H_2}$ and $\eta_{H_3}$ are linearly dependent for every modular triple $(H_1,H_2,H_3)$ of hyperplanes of $M$ with $H_1\cap H_2\cap H_3=F$. Let $N=(MF)\oplus U_{0,F}$, which is a matroid of rank $2$ whose set of hyperplanes is equal to $\cH_F$. This means that $\eta:\cH_F\to T^E$ is a $T$-representation of $N$. By \autoref{thm: correspondence between GP-functions and hyperplane functions}, $\eta$ corresponds to a $T$-matroid $[\varphi]$ that is represented by a Grassmann-Pl\"ucker function $\varphi:E^2\to T$. By \autoref{cor: F-quotients are matroid quotients}, $[\eta_{H}]$ is a matroid quotient of $[\varphi]$ for every $H\in\cH_F$, which establishes the forward implication. 

 Conversely, assume that $[\varphi:E^2\to T]$ is a $T$-matroid such that $[\eta_{H}]$ is a matroid quotient of $[\varphi]$ for every $H\in\cH_F$. Consider a modular triple $(H_1,H_2,H_3)$ of hyperplanes of $M$ with $H_1\cap H_2\cap H_3=F$. By \cite[Theorem~2.16]{Baker-Lorscheid25a}, it suffices to show that
 \[
  \frac{\eta_{H_i}(y)}{\eta_{H_i}(z)} \ = \ \frac{\varphi(x,y)}{\varphi(x,z)}
 \]
 for all $x\in H_i$ and $y,z\in E-H_i$ and $i=1,2,3$, in order to establish the claim that $\eta_{H_1}$, $\eta_{H_2}$ and $\eta_{H_3}$ are linearly dependent. 

 Consider $x\in H_i$ and $y,z\in E-H_i$. Since $[\eta_{H_i}]$ is a quotient of $[\varphi]$, the characterization of matroid quotients in \cite[Theorem~2.17]{Jarra-Lorscheid24} yields
 \[
  \varphi(x,y) \cdot \eta_{H_i}(z) \ - \ \varphi(x,z) \cdot \eta_{H_i}(y) \ + \ \varphi(y,z) \cdot \eta_{H_i}(x) \quad \in \quad N_T.
 \]
 Since $x\in H_i$, we have $f_{H_i}(x) =0$ and thus
 \[
  \varphi(x,y) \cdot \eta_{H_i}(z) \ - \ \varphi(x,z) \cdot \eta_{H_i}(y) \quad \in \quad N_T,
 \]
 which establishes the desired claim for $x\in H_i$ and $y,z\in E-H_i$, which establishes the reverse implication.

 By \autoref{thm: correspondence between GP-functions and hyperplane functions}, the $T$-matroid $[\varphi]$ is uniquely determined by these properties, and its underlying matroid is $M_\varphi=N$.
\end{proof}

This result leads to the following concept of a $(\Lambda^{\leq2},T)$-representation. We denote by $\Lambda^{\leq2}=\big\{F\in\Lambda\,\big|\,\cork F\in\{1,2\}\big\}$ the collection of all flats of $M$ that are of corank $1$ (the hyperplanes) or $2$.

\begin{df}
 A \emph{$(\Lambda^{\leq2},T)$-representation} of $M$ is family of Grassmann-Pl\"ucker functions $\psi=\big\{\psi_F:E^{\cork F}\to T\,\big|\, F\in\Lambda^{\leq2}\big\}$ such that
 \begin{enumerate}[label={($\Lambda$R\arabic*)}]
  \item \label{LR1} $\psi_H:E\to T$ is a hyperplane function for $H$, for every $H\in\cH$
  \item \label{LR2} $[\psi_H]$ is a matroid quotient of $[\psi_F]$ whenever $H\subset F$.
 \end{enumerate}
\end{df}

\begin{thm}\label{thm: cryptomorphism through F-quotients of small corank}
 For every $(\Lambda^{\leq2},T)$-representation $\psi$ of $M$, there is a unique $T$-matroid $[\varphi]$ such that $[\psi_F]=[\varphi_F]$ for every $F\in\Lambda^{\leq2}$. Conversely, every $T$-matroid $[\varphi]$ with underlying matroid $M_\varphi=M$ comes from a $T$-linear representation $\psi$ of $M$ in this way, and $\psi$ is uniquely determined by $[\varphi]$ up to multiplying each function $\psi_F$ by a scalar $a_F\in T^\times$.
\end{thm}

\begin{proof}
 Let $\psi$ be a $(\Lambda^{\leq2},T)$-representation of $M$ and define $\eta:\cH\to T^E$ by $\eta_H=\psi_H$. By \autoref{lemma: characterization of linearly dependent tuples of hyperplane functions}, $\eta_{H_1}$, $\eta_{H_2}$ and $\eta_{H_3}$ are linearly dependent for every modular triple $(H_1,H_2,H_3)$ of hyperplanes of $M$, which shows that $\eta$ is a $T$-representation of $M$. 
 
 Let $[\varphi]$ be the unique $T$-matroid that corresponds to $\eta$ (cf.\ \autoref{thm: correspondence between GP-functions and hyperplane functions}). Then by \autoref{lemma: hyperplane functions as F-quotients}, $[\eta_H]=[\varphi_H]$ for all $H\in\cH$. Since the $T$-representations for $[\varphi]$ are uniquely determined by the classes $[\eta_H]$ (cf.\ \autoref{thm: correspondence between GP-functions and hyperplane functions}), this implies, in turn, that $[\varphi]$ is uniquely determined by these conditions.

 Let $F$ be a corank $2$ flat of $M$. By \autoref{cor: F-quotients are matroid quotients}, $[\varphi_H]$ is a matroid quotient of $[\varphi_F]$ for every hyperplane $H$ that contains $F$. By \autoref{lemma: hyperplane functions as F-quotients}, a $T$-matroid $[\varphi_F]$ with this property is unique, which shows that $[\varphi_F]=[\psi_F]$.
  
 Conversely, a $T$-matroid $[\varphi]$ with underlying matroid $M_\varphi=M$ determines the family of Grassmann-Pl\"ucker functions $\psi=\big\{\psi_F:E^{\cork F}\to T\,\big|\, F\in\Lambda^{\leq2}\big\}$ with $\psi_F=\varphi_F$, which satisfies \ref{LR1} as a consequence of \autoref{thm: correspondence between GP-functions and hyperplane functions} (cf.\ \autoref{subsection: T-representations}) and \ref{LR2} by \autoref{cor: F-quotients are matroid quotients}. Thus $\psi$ is a $(\Lambda^{\leq2},T)$-representation, which, by construction, corresponds to $\varphi$ in the sense of the second claim of the theorem.
 
 The last claim about the uniqueness of $\psi$ (up to scalar multiples) follows from the corresponding claims of \autoref{thm: correspondence between GP-functions and hyperplane functions} for the hyperplane functions $\eta_H$ and by \autoref{lemma: characterization of linearly dependent tuples of hyperplane functions} for the Grassmann-Pl\"ucker functions $\eta_F$ for corank $2$ flats $F$.
\end{proof}

\begin{rem}
 A $(\Lambda^{\leq2},T)$-representation $\psi$ of $M$ can be reinterpreted as a quiver representation in the sense of \cite{Jarra-Lorscheid-Vital24}. Namely, let $Q$ be the quiver with vertices $\Lambda^{\leq2}$ and arrows $F\to H$ whenever $F$ is a corank $2$ flat and $H$ is a hyperplane with $F\subset H$. Since $[\psi_H]$ is a quotient of $[\psi_F]$ if $F\subset H$, the tuple $\big(\{[\psi_F]\mid F\in\Lambda^{\leq2}\},\ \{\id_{E\times E}\mid F\to H\}\big)$ is a $(Q,T)$-representation.
 
 Conversely, a $(Q,T)$-representation $$\big(\{[\psi_F:E^{\cork F}\to T]\mid F\in\Lambda^{\leq2}\},\ \{\id_{E\times E}\mid F\to H\}\big)$$ defines a $(\Lambda^{\leq2},T)$-representation of $M$ if and only if $H=\{e\in E\mid\psi_H(e)=0\}$ for all hyperplanes $H\in\Lambda^{\leq2}$.
\end{rem}

\subsection{Lattices of \texorpdfstring{$T$}{T}-flats}
\label{subsection: lattices of T-flats}

In this section we introduce the notion of a $T$-flat and discuss its properties. This leads eventually to a new cryptomorphic description of $T$-matroid in terms of its $T$-flats. 

\begin{df}
 Let $F\in\Lambda$ be a flat of $M$ of rank $s$. The \emph{$T$-flat of $F$ for $[\varphi]$} is the vector set 
 \[\textstyle
  \cV_{\varphi_F} \ = \ \big\{X\in T^E\,\big|\,\sum_{e\in E} X(e)\cdot \varphi_F(i_1,\dotsc,i_{r-s-1},e)\in N_T\text{ for all }i_1,\dotsc,i_{r-s-1}\in E\big\}
 \]
 of $[\varphi_F]$, which we also denote by $\cV_F$ if $\varphi$ is given by the context. The \emph{lattice of $T$-flats for $[\varphi]$} is the collection $\Lambda_\varphi=\{\cV_F\mid F\in \Lambda\}$ of subspaces of $T^E$.
\end{df}


By \autoref{cor: F-quotients are matroid quotients}, $[\varphi_F]$ is a matroid quotient of $[\varphi_{F'}]$ if $F\subset F'$. Our assumption that $T$ is perfect implies by \autoref{thm: matroid quotients through vector inclusion} that $\cV_F\subset\cV_{F'}$. This means that the partial order of $\Lambda_\varphi$ given by inclusion agrees with the partial order of $\Lambda$, i.e., $\Lambda_\varphi$ agrees with $\Lambda$ as a partially ordered set.

As extremal cases, we find back the vectors $\cV_\varphi=\cV_{\gen\emptyset}$ of $[\varphi]$ as the $T$-flat of the bottom $\gen\emptyset$ of $\Lambda$ and the whole space $T^E=\cV_E$ as the $T$-flat of $E$. The $T$-flat of a hyperplane $H\in\cH$ is the orthogonal complement $\cV_H=\eta_H^\perp$ of the single element $\eta_H\in T^E$ (where we write $X^\perp$ for $\{X\}^\perp$). 

\begin{prop}\label{prop: T-flats are the intersection of T-hyperplanes}
 Let $F\in\Lambda$ be a flat of $M$. Then $\cV_F = \bigcap_{F\subseteq H}\cV_H$.
\end{prop}

\begin{proof}
 Since $\cH_F=\{H\in\cH\mid F\subset H\}$ are the hyperplanes of $M_{\varphi_F}$, we have $\cV_F =\{\eta_H\in T^E\mid H\in\cH_F\}^\perp$ by \autoref{prop: hyperplane description of vectors}, where $\eta:\cH\to T^E$ is a $T$-representation of $M$ that corresponds to $[\varphi]$. Hence
 \[
  \cV_F \ = \ \big\{\eta_H \in T^E\, \big| \, H\in\cH_F \big\}^\perp \ = \ \bigcap_{F\subseteq H} \ \eta_H^\perp \ = \ \bigcap_{F\subseteq H} \ \cV_H. \qedhere
 \]
\end{proof}

This leads to the following definitions. Let $\bV$ be a finite collection of subsets $\cV\subset T^E$ that contains $T^E$ and is closed under intersection. We consider $T^E$ as a poset with respect to inclusion. Since it is finite and closed under intersection, it is a complete lattice. We define the corank of $\cV\in\bV$ as the minimal length $\ell$ of a chain $\cV=\cV_0\subsetneq\dotsc\subsetneq\cV_\ell=T^E$ of proper inclusions with the properties that $\cV_0,\dotsc,\cV_\ell\in\bV$ and that there is no $\cV\in\bV$ with proper inclusions $\cV_{i-1}\subsetneq\cV\subsetneq\cV_i$ for any $i=1,\dotsc,\ell$. In particular, the corank $1$ elements of $\bV$, or its \emph{hyperplanes}, are the maximal elements of $\bV-\{T^E\}$, and the corank $2$ elements of $\bV$ are all maximal elements $\cV\in\bV$ that are neither equal to $T^E$ nor a hyperplane.



For $\cV\in\bV$, we define the \emph{coordinates in $\cV$} as the subset $\Delta_\cV=\{e\in E\mid\delta_e\in \cV\}$ of $E$ where $\delta_e:E\to T$ is the characteristic function of $e$. Let $\Lambda=\{\Delta_\cV\mid \cV\in\bV\}$, which is partially ordered by inclusion and which we call it the \emph{coordinate lattice of $\bV$}. Since it contains $E=\Delta_{T^E}$ and since $\Delta$ preserves intersections, $\Lambda$ indeed a (complete) lattice whose meet is given by intersection. The map $\Delta:\bV\to\Lambda$ is evidently order-preserving. We call the elements of $\Lambda$ its \emph{flats} and define the corank of its flats in the same way as for $\bV$. We denote by $\cH=\cH_\Lambda$ the collection of its corank $1$ flats, or it \emph{hyperplanes}.

If $F\in\Lambda$ has a unique preimage $\cV$ under $\Delta$, then we also write $\cV_F=\cV$. We denote by $\cH_F=\{H\in\cH\mid F\subset H\}$ the collection of hyperplanes of $\Lambda$ that contain $F$ and we define $\Lambda_F=\{F'\in\Lambda\mid F\subset F'\}$ as the sublattice of all flats above $F$. 

\begin{df}\label{definition of lattice of T-flats}
 A \emph{lattice of $T$-flats on $E$} is a finite collection $\bV$ of subsets of $T^E$ such that
 \begin{enumerate}[label={(LF\arabic*)}]
  \item \label{LF1}
  $\bV$ contains $T^E$ and is closed under intersection;
  \item \label{LF2}
  every $\cV\in\bV$ is an intersection of hyperplanes;
  \item \label{LF3}
  the map $\Delta:\bV\to\Lambda$ onto the coordinate lattice $\Lambda$ preserves coranks smaller or equal to $2$;
  \item \label{LF4}
  for every $H\in\cH$, there is an $\eta_H\in T^E$ such that $\cV_H=\eta_H^\perp$;
  \item \label{LF5}
  for every corank $2$ flat $F$ of $\Lambda$, there is a matroid $M_F$ on $E$ whose lattice of flats is $\Lambda_F$ and whose set of cobases is $\cB^\ast_{\cV_F}$.
 \end{enumerate}
\end{df}


\begin{thm}\label{thm: correspondence between T-matroids lattices of T-flats}
 Let $[\varphi:E^r\to T]$ be a $T$-matroid and $\Lambda$ the lattice of flats of its underlying matroid $M$. Then the collection $\bV_\varphi=\{\cV_F\mid F\in\Lambda\}$ of its $T$-flats is a lattice of $T$-flats whose coordinate lattice is equal to $\Lambda$, and thus a lattice, in particular. The map
 \[
  \Phi: \ \big\{\text{$T$-matroids}\big\} \ \longrightarrow \ \big\{\text{lattices of $T$-flats}\big\} 
 \]
 given by $[\varphi]\mapsto \bV_\varphi$ is a bijection.
\end{thm}

The rest of this section is devoted to the proof of this theorem. 

\medskip\noindent\textit{Proof of the first claim.}
We verify that $\bV_\varphi$ is a lattice of flats. First note that it contains $T^E=\cV_{\varphi_E}$. By \autoref{prop: T-flats are the intersection of T-hyperplanes}, every $T$-flat is an intersection of hyperplanes, which establishes \ref{LF2}. Since the $T$-flats in $\bV_\varphi$ are indexed by the lattice of flats of the underlying matroid $M_\varphi$, \autoref{prop: T-flats are the intersection of T-hyperplanes} also implies that $\bV_\varphi$ is closed under intersection thus \ref{LF1}. Axiom \ref{LF4} follows from the fact that $\varphi_F$ is a rank $1$ matroid with a single hyperplane $H$ and that $\eta_H=\varphi$ is a hyperplane function for $H$ that defines a $T$-representation $\eta:\{H\}\to T^E$ that corresponds to $\varphi$. Thus $\cV_H=\eta_H^\perp$ or, equivalently, $\cV_H^\perp=\{a\eta_H\mid a\in T\}$.
 
Since $\eta_H(e)=0$ if and only if $e\in H$, we have $\delta_e\perp\eta_H$ if and only if $e\in H$, and thus $\Delta_{\cV_H}=H$. Since $\Delta:\bV\to\Lambda$ preserves intersections, and the flats in $\Lambda$ correspond to the intersections of hyperplanes, this shows that coordinate lattice of $\bV$ is equal to $\Lambda$. In particular, $\Delta:\bV\to\Lambda$ is an isomorphism of posets, which implies \ref{LF3}. The poset $\Lambda_F$ is the lattice of flats of the underlying matroid $M_\varphi=MF\oplus U_{0,F}$ of $\varphi$. By \autoref{lemma: support bases of a vector set are the cobases of the underlying matroid}, $\cB^\ast_{\cV_\varphi}$ is equal to the set of cobases of $M_\varphi$, which establishes \ref{LF5}.
 
\medskip\noindent\textit{Proof of the second claim.}
The first claim shows that $\Phi$ is a well-defined map. Since $\bV_{\gen\emptyset}=\cV_\varphi$ is the vector set of $[\varphi]$, which uniquely characterizes $[\varphi]$, the map $\Phi$ is injective. The difficult part of this theorem is to establish the surjectivity of $\Phi$.
 
Consider a lattice of $T$-flats $\bV$ with underlying lattice $\Lambda$. Let $\cH$ be the collection of hyperplanes of $\Lambda$. By \ref{LF3}, $\Delta:\bV\to\Lambda$ restricts to a bijection between hyperplanes, i.e., $\Delta^{-1}(H)=\{\cV_H\}$ for every $H\in\cH$. Axiom \ref{LF3} implies further that for every corank $2$ flat $F\in\Lambda$, there is a unique $T$-flat $\cV_F$ of corank $2$ in $\Delta^{-1}(F)$.

A choice of elements $\eta_H\in T^E$ with $\cV_H=\eta_H^\perp$ (which exist by \ref{LF4}) defines a function $\eta:\cH\to T^E$. We aim to show that $\eta$ is a $T$-representation, which determines a $T$-matroid $[\varphi]$ with $\bV_\varphi=\bV$. Since the notion of a $T$-representation makes only sense if $\cH$ is indeed the collection of hyperplanes of a matroid $M$, we establish the following fact as a first step.

\begin{lemma}\label{lemma: Lambda is a lattice of flats}
 The lattice $\Lambda$ is the lattice of flats of a matroid $M$.
\end{lemma}


\begin{proof}
 Let $F\in\Lambda$ be a corank $2$ flat, which is the bottom of the lattice of flats $\Lambda_F$ of some rank $2$ matroid $M_F$ with hyperplanes $\cH_F=\{H\in\cH\mid F\subset H\}$. The covering axiom of flats, applied to $F$ as an element of $\Lambda_F$, yields $E-F=\coprod_{H\in\cH_F} H-F$. By \cite[Thm.\ 2]{Delucchi11}, the covering axiom for corank $2$ flats implies the covering axiom for $\Lambda$, which concludes the proof.
\end{proof}



Let $M$ be the matroid given by \autoref{lemma: Lambda is a lattice of flats}, whose lattice of flats is $\Lambda$ and whose collection of hyperplanes is $\cH$. Consider a hyperplane $H=\Delta_\cV$ of $M$ with $\cV\in\bV$. Then $\eta_H(e)=0$ if and only if $\delta_e\perp\eta_H$, i.e.\ $\delta_e\in\cV$, which means, by definition, that $e\in H$. This shows that $\eta_H:E\to T$ is a hyperplane function for $H$. 

We intend to extend $\{\eta_H\}$ to a $(\Lambda^{\leq2},T)$-representation of $M$. For this, we establish a couple of auxiliar results.

\begin{lemma}\label{lemma: values of a support basis in the corank 2 case}
 Let $F\in\Lambda$ be of corank $2$, $B=E-\{d,e\}$ a support basis for $\cV_F$ and $\{S_i\mid i\in B\}$ a $B$-basis for $\cV_F$. Let $H_d,H_e\in\cH_F$ with $d\in H_d$ and $e\in H_e$. Then
 \[
  S_i(d) \ = \ -\frac{\eta_{H_e}(i)}{\eta_{H_e}(d)} \quad \text{and} \quad S_i(e) \ = \ -\frac{\eta_{H_d}(i)}{\eta_{H_d}(e)}
 \]
 for all $i\in B$.
\end{lemma}

\begin{proof}
 Since $S_i(j)=\delta_{i,j}$ for $j\in B$, the support of $S_i$ is contained in $\{i,d,e\}$. Since $S_i\in\cV_F=\{\eta_H\mid H\in\cH_F\}^\perp$, the term
 \begin{align*}
  \sum_{a\in E} \ S_i(a) \ \cdot \ \eta_{H_d}(a) \ &= \ S_i(i) \ \cdot \ \eta_{H_d}(i) \ + \ S_i(d) \ \cdot \ \eta_{H_d}(d) \ + \ S_i(e) \ \cdot \ \eta_{H_d}(e) \\
  &= \ \eta_{H_d}(i) \ + \  S_i(e) \ \cdot \ \eta_{H_d}(e)
 \end{align*}
 is in $N_T$, where we use that $S_i(i)=1$ and $\eta_d(d)=0$. By axiom \ref{LF5}, $B$ is a cobasis for the matroid $M_F$ with lattice of flats $\Lambda_F$, which means, in particular, that $d$ and $e$ are not in $F$ and contained in distinct hyperplanes. Thus $e\notin H_d$ and $\eta_{H_d}(e)\neq0$. Thus $S_i(e)\cdot\eta_{H_d}(e)=-\eta_{H_d}(i)$, which establishes the claim for $S_i(e)$. The claim for $S_i(d)$ is is proven analogously, with $d$ and $e$ exchanged.
\end{proof}

\begin{lemma}\label{lemma: V is a vector set in the corank 2 case}
 If $\cV\in\bV$ is of corank $2$, then it is the vector set of a $T$-matroid $[\psi_F:E^2\to T]$. 
\end{lemma}

\begin{proof}
 We prove the claim by verifying Anderson's vector axioms, cf.\ \autoref{thm: Anderson's cryptomorphism for vector sets of T-matroids}. Let $F=\Delta_\cV$ be the corank $2$ flat of $\cV$.

 By \ref{LF3} and \ref{LF5}, $\cV$ is the orthogonal complement of $\{\eta_H\mid H\in\cH_F\}$ and as such $T$-invariant. This leaves us with verifying that
 \[
  \cV \ = \ \bigcap_{\substack{B\in\cB^\ast_\cV,\\ \{S_i\}\text{ $B$-basis for $\cV$}}} \gen{S_i\mid i\in B}.
 \]
 Let $B=E-\{d,e\}$ be a support basis for $\cV$ and $\{S_i\}$ a $B$-basis for $\cV$. Consider $X\in\gen{S_i}$, which means that

 \[
  X(c) \ - \ \sum_{i\in B} \ X(i) \ \cdot \ S_i(c) \quad \in \quad N_T
 \]
 for $c\in\{d,e\}$. Replacing $S_i(c)$ by the expressions from \autoref{lemma: V is a vector set in the corank 2 case}, say for $c=d$, shows that this is equivalent with
 \[
  -\eta_{H_e}(d) \ \cdot \ \bigg( X(d) \ - \ \sum_{i\in B} \ X(i) \ \cdot \ \frac{-\eta_{H_e}(d)}{\eta_{H_e}(d)} \bigg) \quad = \quad \sum_{a\in B} \ X(a) \ \cdot \ \eta_{H_e}(d)
 \]
 being in $N_T$, where we use that $\eta_{H_e}(e)=0$ in this equality. This, in turn, means that $X\in\eta_{H_e}^\perp=\cV_{H_e}$. The same reasoning for $c=e$ shows that $X\in\cV_{H_d}$. Thus $\gen{S_i}=\cV_{H_d}\cap\cV_{H_e}$ and $\cV\subset\bigcap\gen{S_i}$.

 Conversely, axiom \ref{LF5} guarantees that every cobasis $B=E-\{d,e\}$ for $\Lambda_F$ is a support basis. Thus given a hyperplane $H\in\cH_F$, we can choose a second distinct hyperplane $H'\in\cH_F$ and $d\in H-F$, $e\in H'-F$, which shows that $\gen{S_i}=\cV_H\cap\cV_{H'}\subset\cV_H$. Since $\cV$ is the intersection of the hyperplanes $\cV_H$ with $H\in\cH_F$, this establishes the reverse inclusion $\bigcap\gen{S_i}\subset\cV$.
\end{proof}

Let $\psi_F=\eta_F:E\to T$. Together with the Grassmann-Pl\"ucker function $\psi_F:E^2\to T$ given by \autoref{lemma: V is a vector set in the corank 2 case}, this defines a $(\Lambda^{\leq2},T)$-representation $\{\psi_F\mid F\in\Lambda^{\leq2}\}$ of $M$: axiom \ref{LR1} holds since $\psi_H=\varphi_H$ is a hyperplane function for $H$, for all $H\in\cH$ axiom \ref{LR2} holds since $\cV_F\subset\cV_H$ whenever $F\subset H$, which means that $[\psi_H]$ is a matroid quotient of $[\psi_F]$. 

By \autoref{thm: cryptomorphism through F-quotients of small corank}, the $(\Lambda^{\leq2},T)$-representation $\{\psi_F\}$ of $M$ determines a $T$-matroid $[\varphi]$ with underlying matroid $M$ with $[\varphi_H]=[\psi_H]=[\eta_H]$ for every $H\in\cH$. Thus $\cV_{\varphi_H}=\eta_H^\perp$. Since the other $T$-flats of $\bV$ correspond to the intersections of the hyperplanes of $\bV$, this shows that $\bV_\varphi=\bV$. This concludes the proof of \autoref{thm: correspondence between T-matroids lattices of T-flats}. \qed

\begin{rem}\label{rem: alternative axioms for lattice of flats}
 The proof of \autoref{thm: correspondence between T-matroids lattices of T-flats} shows in particular that lattices of flats can also characterized as those finite subsets $\bV$ of $T^E$ that satisfy:
 \begin{enumerate}[label=(LF\arabic*)${}^\ast$]
  \item \label{LF1*} 
  $\bV$ contains $T^E$ and is closed under intersection
  \item \label{LF2*} 
  every $\cV\in\bV$ is an intersection of hyperplanes
  \item \label{LF3*} 
  every hyperplane $\cV\in\bV$ is the vector set of a $T$-matroid of rank $1$
  \item \label{LF4*} 
  every $\cV\in\bV$ of corank $2$ is the vector set of a $T$-matroid of rank $2$.
 \end{enumerate}
\end{rem}

\section{Point-line arrangements}
\label{subsection: Point-line arrangements}

In modern language, Tutte's celebrated homotopy theorem (\cite{Tutte58}) is a statement about point-line arrangements (or incidence structures) whose points are labelled by the hyperplanes of a matroid $M$ and whose lines are labelled by (certain) corank $2$ flats of $M$. Our theory extends this notion of point-line arrangements to $T$-matroids in the following way.

We define the \emph{projective space of $T^E$} as the set
\[
 \P(T^E) \ = \ \big(T^E-\{0\}\big) \, \big/ \, T^\times \ = \ \big\{ [X_e]_{e\in E} \, \big| \, X_e\in T\big\},
\]
where $[X_e]$ denotes the class of a nonzero tuple $(X_e)\in T^E$ modulo the diagonal action of $T^\times$ by multiplication. The \emph{projectivization} of a subset $\cV\subset T^E$ is 
\[
 \overline\cV \ = \ \big\{[X_e]\in\P(T^E)\,\big|\, (X_e)\in\cV \big\}. 
\]
A \emph{linear subspaces of $\P(T^E)$} is the projectivization $\overline\cV$ of a linear subspace $\cV\subset T^E$. We define the \emph{(projective) dimension of $\overline{\cV}$} as 
\[
 \dim\overline\cV\ = \ \dim\cV \ - \ 1.
\]
A $0$-dimensional linear subspace of $\P(T^E)$ is nothing else than a singleton of $\P(T^E)$. A \emph{line in $\P(T^E)$} is a linear subspace of dimension $1$. A \emph{point in $\P(T^E)$} is an element $X=[X_e]$ of $\P(T^E)$ and its \emph{support} is $\underline X=\{e\in E\mid X_e\neq0\}$.

\begin{df}
 A \emph{point-line arrangement in $\P(T^E)$} is a finite collection $\cP$ of points in $\P(T^E)$ together with a finite collection $\cL$ of lines in $\P(T^E)$ such that
 \begin{enumerate}[label=(PL\arabic*)]
  \item \label{PL1}
  every line in $\cL$ contains at least two points from $\cP$
  \item \label{PL2}
  $\cC^\ast=\{\underline X\mid X\in\cP\}$ is the cocircuit set of a matroid $M$ 
  \item \label{PL3}
  every \emph{modular} pair of points $X,Y\in\cP$ (i.e., $E-(\underline{X}\cup\underline{Y})$ is a corank $2$ flat of $M$) is contained in a uniquely determined line in $\cL$.
 \end{enumerate}
\end{df}

\begin{thm}\label{thm: correspondence between T-matroids and point-line arrangements}
 Let $[\varphi]$ be a $T$-matroid with underlying matroid $M$ and let $\cH$ be the hyperplane set of $M$. Define 
 \[
  \cP_\varphi \ = \ \bigcup_{H\in\cH} \ \overline\cV^\ast_{\varphi_H} \qquad \text{and} \qquad \cL_\varphi \ = \ \big\{ \, \overline\cV^\ast_{\varphi_F} \, \big| \, F \text{ is a corank $2$ flat of $M$} \big\}.
 \]
 Then $(\cP_\varphi,\cL_\varphi)$ is a point-line arrangement in $\P(T^E)$. The induced map
 \[
  \Psi: \ \ \big\{\text{$T$-matroids on $E$}\big\} \ \ \longrightarrow \ \ \big\{\text{point-line arrangements in $\P(T^E)$}\big\}
 \]
 is a bijection.
\end{thm}

\begin{proof}
 To begin with, note that $\cV^\ast_{\varphi_H}$ is a singleton for $H\in\cH$ and that $\overline\cV^\ast_{\varphi_F}$ is a line in $\P(T^E)$ for every corank $2$ flat $F$ of $M$. Moreover, if $F\subset H$, then $\cV_{\varphi_F}\subset\cV_{\varphi_H}$ and thus $\overline\cV^\ast_{\varphi_H}\subset\overline\cV^\ast_{\varphi_F}$. Since every corank $2$ flat $F$ is contained in at least two hyperplanes, $\overline\cV^\ast_{\varphi_F}$ contains at least two points, which establishes \ref{PL1}.
 
 The covector set $\cV^\ast_{\varphi_H}$ of a hyperplane $H\in\cH$ is spanned by a hyperplane function $\eta_H:E\to T$ for $H$ whose support is a cocircuit of $M$, which implies \ref{PL2}.
 
 A modular pair of points corresponds to covector sets $\cV^\ast_{\varphi_H}$ and $\cV^\ast_{\varphi_{H'}}$ for which $F=H\cap H'$ is a corank $2$ flat of $M$. Thus $\overline\cV^\ast_{\varphi_F}$ is the unique line that contains both $\overline\cV^\ast_{\varphi_H}$ and $\overline\cV^\ast_{\varphi_{H'}}$, which establishes \ref{PL3} and concludes the proof of the first claim.
 
 We turn to the second claim. Since we can recover $\cV_{\varphi_H}$ as the orthogonal complement of a representative $\eta_H:E\to T$ of the unique point $X\in\overline\cV^\ast_{\varphi_H}$, it follows that $\Psi$ is injective. 
 
 In order to establish surjectivity, consider a point-line arrangement $(\cP,\cL)$. For every $X\in\cP$, we choose a representative $\eta_H\in T^E$, which we label with $H=E-\underline X$. Since $\cC^\ast$ is the cocircuit set of a matroid $M$, $\cH=\{E-\underline X\mid X\in\cP\}$ is the set of hyperplanes of $M$. 
 
 Every line in $\cL$ is of the form $\overline\cV^\ast_{\psi_F}$ for a $T$-matroid $[\psi_F:E^2\to T]$ of rank $2$, which we label with the bottom $F\subset E$ of its underlying matroid, i.e., $F=\{e\in E\mid \psi_F(d,e)=0\text{ for all }d\in E\}$. By \ref{PL1}, $\overline\cV^\ast_{\psi_F}$ contains at least two points in $\cP$, which means there are at least two (distinct) hyperplanes $H$ and $H'$ of $M$ such that $\cV_{\psi_F}$ is contained in $\cV_{\eta_H}$ and $\cV_{\eta_{H'}}$. 
 
 This means that $[\eta_H]$ is a matroid quotient of $[\psi_F]$. Since quotients are functorial in the tract morphism $T\to\K$ (\cite[Prop.\ 2.25]{Jarra-Lorscheid24}), the underlying matroid $MH\oplus U_{0,H}$ of $[\eta_H]$ is a quotient of the underlying matroid $MF\oplus U_{0,F}$ of $[\psi_F]$, which means that $H$ (and similarly $H'$) is a hyperplane of $MF\oplus U_{0,F}$. Thus $F=H\cap H'$ is a corank $2$ flat of $M$.
 
 Defining $\psi_H=\eta_H$ for $H\in\cH$ yields thus a $(\Lambda^{\leq2},T)$-representation of $M$. By \autoref{thm: cryptomorphism through F-quotients of small corank}, this determines a $T$-matroid $[\varphi]$ with $[\varphi_F]=[\psi_F]$ if $F$ has corank $1$ or $2$. Thus $(\cP,\cL)=(\cP_\varphi,\cL_\varphi)=\Psi([\varphi])$, as desired.
\end{proof}


\section{Hyperplane arrangements}
\label{section: Hyperplane arrangements}

In this section, we generalize hyperplane arrangements from fields to tracts. Before we embark on this endeavour, we like to warn the reader that the term ``hyperplane'' is used in a different way in this section that is not compatible with its meaning in the earlier sections: the hyperplanes of a hyperplane arrangement over a field correspond to the \emph{rank} $1$ flats of the underlying matroid rather than to the hyperplanes (or \emph{corank} $1$ flats) of the matroid.

We hope that this explanation helps the reader to steer clear of a possible confusion when we use hyperplanes in both senses in the following discussion.

\subsection{Hyperplane arrangements over fields}
\label{subsection: Hyperplane arrangements over fields}

Let $K$ be a field and $V$ an $r$-dimensional vector space. A \emph{hyperplane arrangement in $V$} is a finite collection $\cA=\{H_i\mid i\in E\}$ of pairwise distinct \emph{hyperplanes} $H_i$ of $V$ (which are linear subspaces of codimension $1$) that intersect trivially,\footnote{The restriction that the hyperplanes in $\cA$ intersect trivially is not essential and, in particular, does not change the underlying matroid, but it facilitates our exposition.} i.e., $\bigcap_{i\in E} H_i=\{0\}$. The collection of subsets $F\subset E$ with the property that $j\in F$ whenever $\bigcap_{i\in F}H_i\subset H_j$ forms the family of flats of a matroid $M_\cA$, which is called the \emph{underlying matroid of $\cA$}.

The rank of $M$ is $r$ (with the rank function given by $\br(S)=\codim_V\bigcap_{i\in S}H_i$) and $M$ is a simple matroid (it has no loop since every $H_i$ is a proper subspace of $V$ and it does not have parallel elements because the $H_i$ are pairwise distinct).

In the following, we construct a $K$-matroid associated with $\cA$. We choose a linear functional $\lambda_i:V\to K$ with $\ker\lambda_i=H_i$ for each $i\in E$, which defines a linear map $\lambda:V\to K^E$ with $\lambda(X)=(\lambda_i(X))_{i\in E}$. Since $\bigcap_{i\in E} H_i=\{0\}$, the linear map $\lambda$ is injective and identifies $V$ with an $r$-dimensional linear subspace of $K^E$, which is the covector set $\cV^\ast_\varphi$ of a $K$-matroid $[\varphi:E^r\to K]$ cf.\ \autoref{ex: linear subspaces}.

The \emph{algebraic torus} $T(K)=(K^\times)^E$ acts on $K^E$ by \emph{rescaling} its coordinate axes, and thus on the collection of linear subspaces $\cV$ of $K^E$. We call the $T(K)$-orbit of $\cV$ the \emph{rescaling class of $\cV$}. 

\begin{thm}\label{thm: hyperplane arrangements for fields and the associated K-matroid}
 Let $V$ be an $r$-dimensional $K$-vector space and $\cA=\{H_i\mid i\in E\}$ a hyperplane arrangement in $V$. Let $\lambda:V\to K^E$ be as before and $[\varphi]$ the $K$-matroid with $\cV^\ast_\varphi=\lambda(V)$. Then the following hold:
 \begin{enumerate}
  \item \label{arrangement1} 
  The rescaling class of $\lambda(V)$ is independent of the choices of the $\lambda_i$.
  \item \label{arrangement2} 
  The underlying matroid $M_\cA$ of $\cA$ is equal to $M_\varphi$.
  \item \label{arrangement3} 
  For every flat $F$ of $M_\cA$, the $K$-flat $\cV_F$ of $[\varphi]$ over $F$ satisfies
  \[
   \cV^\perp_F \ = \ \bigcap_{i\in F} \ \lambda(H_i).
  \]
  In particular, $\lambda(H_i)=\cV^\perp_{\{i\}}$.
  \item \label{arrangement4} 
  Every $K$-matroid $[\varphi]$ comes from a hyperplane arrangement in this way namely for $V=\cV^\ast_\varphi$ and $\cA=\{H_i\mid i\in E\}$ with
  \[
   H_i \ = \ \{X\in V\mid X(i)=0\},
  \]
  where we consider $X\in V$ as an element of $K^E$.
 \end{enumerate}
\end{thm}

\begin{proof}
 We begin with \eqref{arrangement1}. The linear functionals $\lambda_i$ are defined by the condition $\ker\lambda_i=H_i$, up to a multiple by a scalar $t_i\in K^\times$. The effect of such scalar multiplies on $\lambda(V)$ is multiplying all elements of $\lambda(V)$ by $(t_i)_{i\in E}$, which is in the same $T(K)$-orbit as $\lambda(V)$. This establishes \eqref{arrangement1}.
 
 We continue with \eqref{arrangement2}. An $r$-subset $B=\{e_1,\dotsc,e_r\}$ of $E$ is a basis of $M_\cA$ if and only if $\bigcap_{i\in B} H_i=\emptyset$. This means that the coordinate projection $K^E\to K^B$ restricts to an isomorphism between $\cV=\lambda(V)$ with $K^B$, which is equivalent with $\varphi(e_1,\dotsc,e_r)\neq0$. Thus $B$ is a basis of $M_\cA$ if and only if it is a basis of $M_\varphi$, which establishes \eqref{arrangement2}.

 In order to verify \eqref{arrangement3}, we choose a $T$-representation $\eta:\cH\to K^E$ of $M_\cA=M_\varphi$ that corresponds to $[\varphi]$, where $\cH$ is the collection of hyperplanes of $M_\cA$. By \autoref{prop: T-flats are the intersection of T-hyperplanes}, $\cV_F=\{\eta_H\mid H\in\cH_F\}^\perp$, where $\cH_F=\{H\in\cH\mid F\subset H\}$. Taking orthogonal complements yields $\cV_F^\perp=\gen{\eta_H\mid H\in\cH_F}_K$, where $\gen{S}_K$ denotes the linear subspace of $K^E$ spanned by $S$. In particular, we have $\cV^\ast_\varphi=\gen{\eta_H\mid H\in\cH}_K$.
 More generally, we have
 \[
  \cV^\perp_F \ = \ \gen{\eta_H\mid H\in\cH_F}_K \ \subset \ \{X\in\cV^\ast_\varphi\mid X(i)=0\text{ for all }i\in F\} \ = \ \bigcap_{i\in F} \ \lambda(H_i).
 \]
 Since $\dim\cV^\ast_\varphi=r$, we conclude that
 \[
  \dim \cV^\perp_F \ = \ \cork F \ = \ \dim\cV^\ast \ - \ \rk F \ = \ \dim \bigcap_{i\in F} \ \lambda(H_i),
 \]
 which establishes the equality between these two linear subspaces, and thus \eqref{arrangement3}.
 
 Claim \eqref{arrangement4} is evident if we choose the linear functionals as the characteritic functions $\lambda_i$ of the standard basis elements $\delta_i:E\to K$ of $K^E$, which induce the tautological embedding $\lambda:\cV^\ast_\varphi\to K^E$.
\end{proof}

\subsection{The canonical hyperplane arrangement of a \texorpdfstring{$T$}{T}-matroid}
\label{subsection: the hyperplane arrangement of a T-matroid}

\autoref{thm: hyperplane arrangements for fields and the associated K-matroid} paves the way to generalize hyperplane arrangements from fields to tracts: in essence (i.e., up to rescaling), a hyperplane arrangement $\cA=\{H_i\mid i\in E\}$ in $V$ determines an embedding $\lambda:V\to K^E$ such that the hyperplanes $H_i$ and their intersections are equal to the intersections of $V$ with the coordinate subspaces $Z_F=\{X\in K^E\mid X(i)=0\text{ for }i\in F\}$ of $K^E$. While there is not a good notion of an \emph{abstract} linear spaces for an arbitrary tract $T$, we have a satisfactory notion of a linear subspace, so that the latter viewpoint makes sense for $T$-matroids. 

We assume in this section that $T$ is a perfect tract. We say that a $T$-matroid is \emph{simple} if its underlying matroid is simple.

\begin{df}
 Let $[\varphi:E^r\to T]$ be a simple $T$-matroid. The \emph{canonical hyperplane arrangement of $[\varphi]$} is the collection $\cA=\{H_i\mid i\in E\}$ of subspaces $H_i=\{X\in\cV^\ast_\varphi\mid X(i)=0\}$ of the covector set $\cV^\ast_\varphi\subset T^E$ of $[\varphi]$.
\end{df}

\begin{prop}\label{prop: canonical hyperplane arrangement}
 Let $[\varphi:E^r\to T]$ be a simple $T$-matroid with underlying matroid $M$ and $\cA=\{H_i\mid i\in E\}$ its canonical hyperplane arrangement. For $S\subset E$, let $H_S=\bigcap_{i\in S} H_i$. Then $F\subset E$ is a flat of $M$ if and only if $F=\{i\in E\mid H_F\subset H_i\}$. In this case, $H_F$ is the orthogonal complement of the $T$-flat $\cV_F$ of $[\varphi]$ over $F$. In particular, $H_F\subset T^E$ is a linear subspace of dimension $\cork F$.
\end{prop}

\begin{proof}
 We establish the claims of the proposition in reverse order. Our first task is to verify Anderson's axioms of a linear subspace for $H_S$, which requires us to determine the support bases of $H_S$ as a first step. Composing a covector $X:E\to T$ in $\cV^\ast_\varphi$ with the unique tract morphism $f:T\to \K$ yields a covector $f\circ X:E\to\K$ of the $\K$-matroid $[f\circ\varphi]$ by \cite[Prop.\ 4.6]{Anderson19}, which has the same underlying matroid $M$ as $[\varphi]$. The zero set $\upZ(f\circ X)=\{e\in E\mid f\circ X(e)=0\}$ of $f\circ X$ is a flat of $M$ (cf.\ \autoref{ex: linear subspaces}) and equal to the zero set $\upZ(X)=\{e\in E\mid X(e)=0\}$ of $X$.
 
 Consider a nonzero element $X\in H_S$ with minimal support. Then $\upZ(X)$ contains $S$ and is a proper flat of $M$. Thus it is contained in a hyperplane $H$ of $M$ or, equivalently, $\underline H\subset\underline X\neq E$. Since $X$ has minimal support, we conclude that $\underline H=\underline X$, which shows that the set of nonzero elements of $H_S$ with minimal support is the cocircuit set $\cC^\ast_F$ of $[\varphi_F]$, whose orthogonal complement is, by definition, $\cV_F$. As conclusion, the set $\cB^\ast_{H_S}$ of support bases of $H_S$ agrees with the set $\cB_{MF\oplus U_{0,F}}$ of bases of the underlying matroid $MF\oplus U_{0,F}$ of $[\varphi_F]$.

 Given an arbitrary element $X\in H_S$, we aim to show that for every $B\in\cB^\ast_{H_S}$ and every $B$-basis $\{S_i\mid i\in B\}$ for $H_S$, we have $X\in\gen{S_i}$. In order to verify this, we extend this $B$-basis to a $\widehat B$-basis for $\cV^\ast_\varphi$ in the following.

 Choose a maximal linear independent subset $J\subset S$ of $M$. Then for any basis $B\subset E$ of $MF\oplus U_{0,F}$, which does not contain any element of $F$ and thus is a basis of $MF$, the union $\widehat B=B\cup J$ is a basis of $M$, by the definition of $MF$. 
 
 Let $\{S_i\mid i\in B\}$ be a $B$-basis of $H_S$, i.e., $S_i(j)=\delta_{i,j}$ for all $i,j\in B$. Since $S_i\in H_S$, we have $S_i(e)=0$ for all $e\in S$. Let $\widehat B^c=E-\widehat B$, which is a cobasis of $M$ and $i\in J$. Then $\widehat B^ci$ contains a fundamental cocircuit $C_i$ of $M$, which is the support of a cocircuit $S_i\in\cV^\ast_\varphi$ of $[\varphi]$, which we can assume to satisfy $S_i(i)=1$. Since $B\cup J\not\subset\widehat B^c$, we have $S_i(j)=\delta_{i,j}$ for all $i,j\in \widehat B$. This extends the $B$-basis $\{S_i\mid i\in B\}$ for $H_S$ to a $\widehat B$-basis $\{S_i\mid i\in\widehat B\}$ of $\cV^\ast_\varphi$. 
 
 We return to our task of showing that $X\in\gen{S_i\mid i\in B}$. Since $X\in\cV^\ast_\varphi$, we have $X\in\gen{S_i\mid i\in\widehat B}$. Thus
 \[
  X(e) \ - \ \sum_{i\in\widehat B} \ X(i) \ \cdot \ S_i(e) \ \ \in \ \ N_T
 \]
 for all $e\in E-\widehat B$. Since $X(i)=0$ for $i\in S$ and thus, in particular for $i\in J$, we get 
 \[
  X(e) \ - \ \sum_{i\in B} \ X(i) \ \cdot \ S_i(e) \ \ \in \ \ N_T,
 \]
 which shows that $X\in\gen{S_i\mid i\in B}$, as desired. This shows that $H_S$ is a linear subspace of $T^E$, which establishes the last claim of the proposition.
 
 As a consequence, $H_S$ is the covector set of a $T$-matroid $[\psi_S]$. As argued above, the support minimal nonzero elements of $H_S$ are the cocircuits $C\in\cV^\ast_\varphi$ of $[\varphi]$ with $S\subset \upZ(C)$, which are thus the cocircuits of $[\psi_S]$. Since the intersection of the vanishing sets $\upZ(C)$ of all cocircuits $C$ of $[\varphi]$ with $S\subset \upZ(C)$ is the flat $F$ generated by $S$, we see that the cocircuits of $H_S$ coincide with the cocircuits of $[\varphi_F]$. This shows in particular that $H_F=\cV_F^\perp$ for $S=F$, which establishes the second claim of the proposition.
 
 Moreover, we deduce from the identification $H_S=\cV^\perp_F$ that $H_S\subset H_i=\cV^\perp_{\varphi_{\{i\}}}$ if and only if $i\in F$ and thus $F=\{i\in E\mid H_S\subset H_i\}$ equals $S$ if and only if $S=F$ is a flat. This establishes the first claim.
\end{proof}

\section{Application to tropical linear spaces}
\label{section: Application to valuated matroids}

In contrast to linear subspaces of vector spaces over fields, which can be described (and classified) in terms of bases, tropical linear spaces escape such a simple access. The class of linear spaces that can be analogously described by a basis is called Stiefel tropical linear spaces. Fink and Rincon show in \cite{Fink-Rincon15} that every Stiefel tropical linear space is representable, and thus form a low dimensional part of the space of all tropical linear spaces (with fixed rank and ground set), which is called the Dressian.

Tropical linear spaces have several equivalent characterization. In our terminology, the original definition due to Speyer (\cite{Speyer08}) considers tropical linear space as the covector set $\cV_\varphi\subset \T^E$ of a $\T$-matroid $[\varphi:E^r\to \T]$, a.k.a.\ valuated matroid, which is the same as a linear subspace of $\T^E$ in the sense of this paper. Equivalently a subset $\cV$ of $\T^E$ that is closed under \emph{tropical linear combinations}, i.e.\ under scalar multiplication by $\T$ and under the coordinatewise maximum of its elements, such that $\log\cV\subset\overline\R^E$ is a tropical variety of weight $1$, where $\overline\R=\R\cup\{\infty\}$ cf.\ \cite{Hampe15}. 



\subsection{Point-line arrangements}

In low dimensions, the latter description of (the projectivization of) a tropical linear space is particularly simple: the projectivization of a unidimensional tropical linear space $\cV\subset\T^E$ is a point in $\P(T^E)$. The projectivization of a two-dimensional tropical linear space $\cV\subset\T^E$ is a tropical curve in $\P(T^E)$. More explicitly, if we decompose the projective space as a set into
\[
 \P(T^E) \ = \ \bigsqcup_{\emptyset\neq S\subset E} (T^\times)^S/\T,
\]
then the projectivization $\overline\cV$ is the closure of $C=\overline\cV\cap (T^\times)^S/\T$ (in the topology on $\P(\T^E)$ induced by the order topology on $\T=\R_{\geq0}$) for some $S$ (the non-loops of the supporting matroid of $\cV$), and $C$ is a tropical curve in $(T^\times)^S/\T$.

Thus a point-line arrangement in $\P(\T^E)$ is an arrangement of points $\cP$ and tropical lines $\cL$ in $\P(\T^E)$ that satisfy the three axioms of a point-line arrangement:
\begin{enumerate}
 \item every line in $\cL$ contains at least $2$ points in $\cP$;
 \item $\cC^\ast=\{\underline X\mid X\in\cP\}$ is the cocircuit set of some matroid $M$;
 \item every modular pair $X,Y\in\cP$ is contained in a unique line in $\cL$.
\end{enumerate}
This perspective might have the potential to construct valuated matroids with particular properties.

\subsection{Hyperplane arrangement}

The canonical hyperplane arrangement of a valuated matroid $M$ is given by the collection $\cA$ of hyperplanes $H_i=\{X\in\cV^\ast_M\mid X(i)=0\}$ in the tropical linear space $\cV_M^\ast$ associated with $M$, which are tropical linear spaces themselves.

\subsection{Lattice of flats}

A lattice of $\T$-flats is determined as the collection $\bV$ of intersections of a finite set of tropical hyperplanes $\cV_H$ of $\T^E$ (i.e., the bend locus of a ``tropical linear form'' $\eta_H:E\to\T$) for which every maximal linear subspace in $\bV-\{\T^E,\text{hyperplanes}\}$ has codimension $2$.

\begin{small}
 \bibliographystyle{plain}
 \bibliography{bib}
\end{small}

\end{document}